\documentclass[11pt]{article}
\usepackage{amsfonts}
\usepackage{color}
\usepackage{amsmath,amssymb,comment}
\usepackage[hidelinks]{hyperref}

\newtheorem{theo}{Theorem}[section]
\newtheorem{lem}[theo]{Lemma}
\newtheorem{cor}[theo]{Corollary}
\newtheorem{prop}[theo]{Proposition}

\newtheorem{defi}[theo]{Definition}

\newcommand{\mysection}[1]{\section{#1} \setcounter{equation}{0}}
\newcommand{\proof}{{\sc Proof.} \quad}
\newcommand{\proofc}{{\sc Proof} \ }
\newcommand{\be}{\begin{equation} \label}
\newcommand{\ee}{\end{equation}}
\newcommand{\bea}{\begin{eqnarray}\label}
\newcommand{\eea}{\end{eqnarray}}
\newcommand{\bas}{\begin{eqnarray*}}
\newcommand{\eas}{\end{eqnarray*}}
\newcommand{\bit}{\begin{itemize}}
\newcommand{\eit}{\end{itemize}}
\newcommand{\qed}{\hfill$\Box$ \vskip.2cm}
\newcommand{\nn}{\nonumber}
\newcommand{\R}{\mathbb{R}}
\newcommand{\N}{\mathbb{N}}
\newcommand{\pO}{\partial\Omega}

\newcommand{\dist}{{\rm dist} \, }
\newcommand{\supp}{{\rm supp} \, }

\newcommand{\hra}{\hookrightarrow}
\newcommand{\io}{\int_\Omega}

\newcommand{\na}{\nabla}
\newcommand{\Del}{\Delta}
\newcommand{\del}{\delta}

\newcommand{\lam}{\lambda}
\newcommand{\Lam}{\Lambda}

\newcommand{\pa}{\partial}
\newcommand{\bom}{\overline{\Omega}}
\newcommand{\Om}{\Omega}

\newcommand{\ov}{\overline}

\newcommand{\hs}{\hspace*}

\newcommand{\vp}{\varphi}
\newcommand{\lbal}{\left\{ \begin{array}{l}}
\newcommand{\lball}{\left\{ \begin{array}{ll}}
\newcommand{\ear}{\end{array} \right.}

\newcommand{\abs}{\\[5pt]}

\newcommand{\adb}{\allowdisplaybreaks}

\newcommand{\Fp}{{\mathcal{F}}_{\vp}}
\newcommand{\Ep}{{\mathcal{E}}_{\vp}}
\newcommand{\B}{{\mathcal{B}}}

\begin{document}
\adb
%\enlargethispage{10mm}
%
%
\title{Refined temporal asymptotics near blow-up points\\
in the planar Keller-Segel system}
\author{
Frederic Heihoff\footnote{fheihoff@math.uni-paderborn.de}\\
{\small Universit\"at Paderborn, Institut f\"ur Mathematik}\\
{\small 33098 Paderborn, Germany}
\and
Michael Winkler\footnote{michael.winkler@math.uni-paderborn.de}\\
{\small Universit\"at Paderborn, Institut f\"ur Mathematik}\\
{\small 33098 Paderborn, Germany}}
\date{}
\maketitle
\begin{abstract}
\noindent 
For the Keller-Segel system
\[
  \left\{\,
  \begin{aligned}
    u_t &= \Delta u - \nabla \cdot ( u \nabla v ), \\
    v_t &= \Delta v - v + u
  \end{aligned}
  \right.
  \tag{$\star$}
\]
posed in a planar domain $\Omega$ with Neumann boundary conditions, the existence of classical solutions blowing up at some finite time $T$ has long been established. In fact, it has been shown that for every blow-up point $x$ the quantity $\int_{B_R(x)\cap\Omega} u(\cdot,t )\ln(u(\cdot, t))$ is unbounded as $t\nearrow T$ for all $R > 0$ even though the global mass of $u$ is always conserved.\abs
The present manuscript provides some quantitative information on the behavior of such localized $L\log L$ expressions
by asserting the existence of $\delta_0=\delta_0(\Omega)>0$ such that any solution to the Neumann problem for ($\star$)
blowing up at time $T\in (0,\infty)$ satisfies
\[
  \limsup_{t\nearrow T} \frac{1}{\ln\frac{T}{T-t}}\int_{B_R(x)\cap\Omega} u(\cdot, t)\ln(u(\cdot, t)) \ge \delta_0
  \tag{$\star\star$}
\]
for all $R > 0$ at each blow-up point $x$.
This confirms a certain universality property of the blow-up mechanism seen in the particular examples of radial
collapsing solutions constructed in the seminal work \cite{HV}, 
especially also beyond the realm of symmetry;
apart from that, along with a consequence of ($\star\star$) on the corresponding asymptotics of similarly 
localized $L^p$ norms of $u$ for $p\in (1,\infty]$, this
provides some extension of a known result on non-degeneracy of blow-up points that has concentrated 
on the choice $p=\infty$ here.\abs
\noindent {\bf Key words:} chemotaxis; blow-up; asymptotics; localized energy\\
 {\bf MSC 2020:} 92C17 (primary); 35B40, 35B44, 35K10 (secondary)
% 92C17 -	Cell movement (chemotaxis, etc.)
% 35B40 - Asymptotic behavior of solutions to PDEs
% 35B44 - Blow-up in context of PDEs
% 35K10 - Second-order parabolic equations

\end{abstract}
\newpage
\section{Introduction}\label{intro}
Understanding blow-up phenomena forms a central theme in the analysis of chemotaxis systems. 
Subsequent to the introduction of the classical Keller-Segel model in 1970 (\cite{KS}),
predictions concerning the occurrence and the quality of singularity formation (\cite{nanjundiah}, \cite{horstmann})
have given rise to mathematical challenges that have significantly stimulated the analysis of parabolic systems
and their steady states during the past decades. 
In line with the 		%intricate structure
considerable intricacy of taxis-type cross-diffusion that shows up as a determinant model ingredient,
already the mere detection of blow-up seems to necessitate substantial efforts, and despite remarkably vivid activities
in the literature the picture in this regard seems yet far from complete.
After all,
%When posed in smoothly bounded and simply connected planar domains $\Om$, for instance,
the Neumann problem for the classical Keller-Segel system,
\be{0}
	\lball
	u_t = \Del u - \na \cdot (u\na v), 
	\qquad & x\in\Om, \ t\in (0,T), \\[1mm]
	v_t = \Del v - v + u,
	\qquad & x\in\Om, \ t\in (0,T), \\[1mm]
	\frac{\pa u}{\pa\nu}=\frac{\pa v}{\pa\nu}=0,
	\qquad & x\in\pO, \ t\in (0,T), \\[1mm]
	u(x,0)=u_0(x), \quad v(x,0)=v_0(x),
	\qquad & x\in\Om,
	\ear
\ee
is known to admit, when posed in smoothly bounded and simply connected planar domains $\Om$, 
some solutions blowing up with respect to the spatial $L^\infty$ norm of their first solution component 
either in finite or infinite time (\cite{horstmann_wang}); in radially symmetric settings, this time of explosion 
has been found to be actually finite for suitably chosen initial data (\cite{HV}).
While no such explosions occur in one-dimensional domains (\cite{osaki_yagi}), in three- or higher-dimensional settings
only some radial blow-up solutions seem to have been constructed so far (\cite{win_JMPA}).
Sufficient conditions for the absence of blow-up have been identified in \cite{NSY} and \cite{cao_small}, and 
extensions to Cauchy problems in $\R^n$ for $n\ge 2$ can be found in \cite{schweyer} and \cite{win_NON}.\abs
With regard to the qualitative behavior of unbounded solutions near the time of their explosion, however,
detailed descriptions of blow-up asymptotics could so far be achieved only for parabolic-elliptic simplifications of (\ref{0}).
In particular, finite-time blow-up in two-dimensional domains, as detected for sufficiently concentrated initial data
at large mass levels in \cite{nagai2001} and \cite{biler} (see also \cite{jaeger_luckhaus} and \cite{suzuki_book}), 
can only occur at finitely many points at which a Dirac mass is approached
(\cite{senba_suzuki_ADE2001}, 
\cite{senba_suzuki_ADE2003}, \cite{senba_NA2007}; 
see also \cite{delpino_arxiv}); 
in higher-dimensional relatives, integrable point sigularities may arise (\cite{souplet_win}).
Radial blow-up in associated Cauchy problems in $\R^n$, $n\ge 2$,
exclusively occurs at the center of symmetry (\cite{win_JEMS}), 
with a universal characterization of the corresponding non-selfsimilar asymptotics 
available when $n=2$ (\cite{mizo}),		%; cf.~also \cite{collot_CPAM2022} for nonradial stability), 
and with quite precise descriptions
gained for some examples of infinite-time bubbling exhibited by some mass-critical solutions when $n=2$
(\cite{DDelPDMW}, \cite{ghoul_CPAM2018}), and of 
selfsimilar or also of certain collapsing-ring solutions in the case $n=3$ (\cite{glogic}, \cite{collot_JFA2023}).
Spatial blow-up profiles of radial solutions to such quasistationary variants of (\ref{0})
allow for fairly precise pointwise estimates in frameworks of radially decreasing functions (\cite{souplet_win}, \cite{xueli_bai}),
with subtle deviations recently detected in the absence of radial monotonicity (\cite{zaag}).
For a small selection of findings on bubbling phenomena in fully stationary versions of (\ref{0}) and some close
relatives, we refer to \cite{agudelo_pistoia}, \cite{delpino_pistoia_vaira}, \cite{bonheure_JMPA}, \cite{ni_takagi_DMJ1993}
and \cite{delpino_JEMS2014}.\abs
{\bf Specifying the goal: Examining typicality of a particular blow-up mechanism.} \quad
%{\bf How far is the blow-up mechanism described by Herrero and V\'elazquez prototypical?} \quad
%
%
Markedly less comprehensive seems the information on qualitative behavior that is available for the fully parabolic system (\ref{0}).
Especially, it seems to date widely open to which extent the particular blow-up mechanism discovered in the seminal paper \cite{HV},
and later on discussed with respect to stability in \cite{schweyer},
can be regarded as a prototypical representative for singularity formation in the two-dimensional version of (\ref{0}).
In fact, the construction in \cite{HV} addresses the case when $\Om=B_{R_0}(0)\subset\R^2$ for some $R_0>0$, and yields $T>0$ and $A>0$
as well as a radially symmetric classical solution $(u,v)$ of (\ref{0}) which satisfies
\be{proto}
	u(x,t)=\frac{8}{\lam^2(t)} \cdot W\Big( \frac{x}{\lam(t)} \Big) \cdot \big( 1+o(1)\big),
\ee
uniformly in the set $\{(x,t)\in \Om\times (0,T) \ | \ |x|\le \lam(t)\}$ as $t \nearrow T$, where $W(\xi):=\frac{1}{(|\xi|^2+1)^2}$ for $\xi\in\R^2$ and
\be{lam}
	\lam(t)=A\sqrt{T-t} \cdot e^{-\sqrt{\frac{|\ln (T-t)|}{2}}} \cdot \big(1+o(1)\big) \qquad\text{ as }
	 t\nearrow T.
\ee
In particular, this solution has the property that the expression $I(t):=(T-t)\|u(\cdot,t)\|_{L^\infty(\Om)}$ is bounded from below
by a positive constant near $t=T$; in fact, we even have $I(t)\to +\infty$ in this case, meaning that $(u,v)$ 
undergoes a so-called type-II blow-up.
Although it seems unknown whether solutions other than the ones obtained in \cite{HV} and \cite{schweyer} may exhibit type-II explosions,
in \cite{mizoguchi_souplet} it could at least be asserted that any blow-up phenomenon in (\ref{0}) that occurs at time $T<\infty$ 
must be such that whenever $x_0\in\bom\subset\R^2$ is a blow-up point according to Definition \ref{db} below, we have
\bas
	\limsup_{t\nearrow T} \Big\{ (T-t) \|u(\cdot,t)\|_{L^\infty(B_R(x_0)\cap\Om)} \Big\} \ge \del
	\qquad \mbox{for all }  R>0
\eas
with some positive constant $\del$ depending on $\Om$ only.
While yielding a depiction which up to the subalgebraic correction in (\ref{lam}) is optimal for the solution in (\ref{proto})
at the point $x=0$ where the spatial maximum is attained, this result evidently does not characterize the behavior in any neighborhood
of this point, and hence appears unable to capture possibly more subtle facets of spatio-temporal dynamics.
Specifically, for each $p\in (1,\infty)$ the solution in (\ref{proto}) can be seen to actually satisfy
\bas
	\|u(\cdot,t)\|_{L^p(\{ |x|<\lam(t)\})} \ge C (T-t)^{-\frac{p-1}{p}} \cdot e^{\frac{p-1}{p}\sqrt{2|\ln (T-t)|}}
\eas
for some $C>0$ and all $t<T$ sufficiently close to $T$, whence in particular, since $\lam(t)\to 0$ as $t\nearrow T$,
\be{lp}
	\limsup_{t\nearrow T} \Big\{ (T-t)^{\beta(p)} \|u(\cdot,t)\|_{L^p(B_R(0))} \Big\} \ge \del
	\qquad \mbox{for all }  R\in (0,R_0)
\ee
with some $\del>0$, and with the number $\beta(p):=\frac{p-1}{p}$ which, in line with an evident mass conservation property in (\ref{0}),
satisfies $\beta(p)\to 0$ as $p\to 1$. Apart from that, (\ref{proto}) may hint on possible behavior also of expressions 
significantly closer to spatial $L^1$ norms; for instance, there exists $\del>0$ such that for the solution from \cite{HV} we have
\be{log}
	\int_{B_R(0)} u(\cdot,t) \ln u(\cdot,t) \ge \del \cdot \ln \frac{1}{T-t}
	\qquad \mbox{for all }  R\in (0,R_0)
\ee
for any $t<T$ in suitably small distance to $T$.\abs
To the best of our knowledge, quantitative information such as that in (\ref{lp}) and (\ref{log}) so far seems available only
for the particular solutions from \cite{HV} and \cite{schweyer};
for more general solutions blowing up at time $T$, only some fairly coarse properties related to space-time asymptotics 
have been asserted up to now, such as the fact that
at each blow-up point $x_0\in\bom\subset\R^2$ we have (\cite{NSS2000})
\be{coarse}
	\int_{B_R(x_0)\cap\Om} u(\cdot,t) \ln u(\cdot,t) \to + \infty
	\quad \mbox{as } t\nearrow T
	\qquad \mbox{for all } R>0, 
\ee
and that, when $\Om\subset\R^n$ with $n\ge 2$ but then only at a spatially global level, 
the family $(u^\frac{n}{2}(\cdot,t))_{t\in (0,T)}$ cannot be uniformly integrable over $\Om$ (\cite{cao_arxiv}).\abs
{\bf Main results.} \quad
The purpose of the present manuscript is to discuss how far properties such as those in (\ref{lp}) and even (\ref{log})
are shared by arbitrary non-global solutions to (\ref{0}) in the case when $\Om$ is a bounded planar domain.
In order to prepare the formulation of our main results in this regard, let us recall the following outcome
of standard approaches toward the construction of maximally extended solutions (\cite{horstmann_win}):
\begin{prop}\label{prop0}
  Let $\Om\subset\R^2$ be a bounded domain with smooth boundary, and suppose that
  \be{init}
	\lbal
	u_0 \in C^0(\bom) \mbox{ is nonnegative with $u_0\not\equiv 0$, \qquad and that} \\[1mm]
	v_0 \in \bigcup_{q>2} W^{1,q}(\Om) \mbox{ is nonnegative.}
	\ear
  \ee
  Then there exist $T\in (0,\infty]$ as well as uniquely determined functions
  \be{reg}
	\lbal
	u\in C^0(\bom\times [0,T)) \cap C^{2,1}(\bom\times (0,T))
	\qquad \mbox{and} \\[1mm]
	v \in \bigcup_{q>2} C^0([0,T);W^{1,q}(\Om)) \cap C^{2,1}(\bom\times (0,T))
	\ear
  \ee
  such that $u>0$ and $v\ge 0$ in $\Om\times (0,T)$,
  that $(u,v)$ solves (\ref{0}) in the classical sense, and that
  \be{ext}
	\mbox{if $T<\infty$, \quad then \quad}
	\limsup_{t\nearrow T} \|u(\cdot,t)\|_{L^\infty(\Om)}=\infty.
  \ee
\end{prop}
Henceforth restricting ourselves to the analysis of non-global solutions, we shall refer to the standard notion of "blow-up point",
as specified in the following.
\begin{defi}\label{db}
  Let $\Om\subset\R^2$ be a bounded domain with smooth boundary, assume (\ref{init}), and suppose that for $(u,v)$ and $T$ as in
  Proposition \ref{prop0} we have
  $T<\infty$. Then we call $x_0\in\bom$ a {\em blow-up point} of $(u,v)$ if 
  \be{db1}
	\limsup_{t\nearrow T} \|u(\cdot,t)\|_{L^\infty(B_R(x_0)\cap\Om)} =\infty
	\qquad \mbox{for all } R>0.
  \ee
  We moreover let $\B$ denote the set of all blow-up points of $(u,v)$.
\end{defi}
{\bf Remark.} \quad
It can readily be verified that in the situation of Definition \ref{db}, $x_0\in\bom$ is a blow-up point of $(u,v)$ if and only if
there exist $(x_k)_{k\in\N}\subset\Om$ and $(t_k)_{k\in\N} \subset (0,T)$ such that $x_k\to x_0, t_k\nearrow T$ and
$u(x_k,t_k)\to +\infty$ as $k\to\infty$. 
Accordingly, for the corresponding blow-up set we have $\B\ne\emptyset$.\abs
Our main result now asserts that a property of the form in (\ref{log}) can indeed be observed near any blow-up point of arbitrary
non-global solutions:
\begin{theo}\label{theo11}
  Let $\Om\subset\R^2$ be a bounded domain with smooth boundary.
  Then one can choose $\del_0=\del_0(\Om)>0$ in such a way that if $u_0$ and $v_0$ are such that (\ref{init}) holds,
  and that in Proposition \ref{prop0} we have $T<\infty$, then each $x_0\in\B$ has the property that
  \be{11.1}
	\limsup_{t\nearrow T} \frac{1}{\ln\frac{T}{T-t}} \cdot \int_{B_R(x_0)\cap\Om} u(\cdot,t)\ln u(\cdot,t) >\del_0
	\qquad \mbox{for all } R>0;
  \ee
  that is, given any such $x_0$ and an arbitrary $R>0$ one can find $(t_k)_{k\in\N} = (t_k^{(x_0,R)})_{k\in\N} \subset (0,T)$
  fulfilling $t_k\nearrow T$ as $k\to\infty$ and
  \bas
	\int_{B_R(x_0)\cap \Om} u(\cdot,t_k) \ln u(\cdot,t_k) \ge \del_0 \cdot \ln \frac{T}{T-t_k}
	\qquad \mbox{for all } k\in\N.
  \eas
\end{theo}
As a fairly simple consequence, any such blow-up point actually also enjoys a property resembling that in (\ref{lp}):
\begin{cor}\label{cor12}
  Let $\Om\subset\R^2$ be a bounded domain with smooth boundary, and let $\del_0=\del_0(\Om)$ be as in Theorem \ref{theo11}.
  Then whenever (\ref{init}) holds and in Proposition \ref{prop0} we have $T<\infty$, then there exists $C(u_0,v_0)>0$ such
  that for arbitrary $x_0\in\B$ and any $p\in (1,\infty]$,
  \be{12.1}
	\limsup_{t\nearrow T} \bigg\{ (T-t)^{\frac{p-1}{p} \cdot \frac{\del_0}{m}} \|u(\cdot,t)\|_{L^p(B_R(x_0)\cap\Om)} \bigg\}
	\ge C(u_0,v_0)
	\qquad \mbox{for all } R>0,
  \ee
  where $m:=\io u_0$, and where $\frac{p-1}{p}:=1$ if $p=\infty$.
\end{cor}
{\bf Main ideas.} 
To prove Theorem \ref{theo11}, we will rely on a
proof by contradiction. More specifically for a significant portion of this manuscript, we will work under the assumption that for a solution $(u,v)$ to (\ref{0}) blowing up at finite time $T$ the blow-up asymptotics outlined in (\ref{11.1}) are not satisfied at some blow-up point $x_0 \in \B$ for some radius $R > 0$. This means more concretely that for any $\delta > \delta_0$ there must thus exist a time $t_0 \in [0,T)$ such that 
\be{mi.1}
  \int_{B_R(x_0)\cap\Omega} u \ln(u + e) \leq \delta \cdot \ln\frac{T}{T-t} \qquad \mbox{for all } t \in (t_0, T).
\ee
Using this upper estimate, we then show that $u$ in fact must have already have been uniformly bounded up to the blow-up time $T$ in an appropriate sense and thus rule out $x_0$ as a blow-up point, which leads us to the promised contradiction. We will now give a brief sketch of how the estimate in (\ref{mi.1}) can be leveraged toward this end. In this context, $\vp$ will typically be a cut-off function localized around $x_0$. Crucially, it will not necessarily always be the same cut-off function as in later steps of our argument it is often necessary to localize successively more and more to make use of information gained in previous steps. 
\abs
At the very core of our argument lies a localized version of an energy functional commonly encountered in the study of the Keller-Segel system (cf. Lemma \ref{lem_E}), namely
\be{E0}
  \mathcal{E}_\vp(t) := \io \vp^4 u(\cdot, t)(\ln u(\cdot, t) - 1) + \io \vp^4 |\nabla v(\cdot, t)|^2.
\ee
If we can show that this functional is uniformly bounded in time, blow-up is already ruled out for the point $x_0$ due to (\ref{coarse}) and our argument is complete.
\abs
Before we dive into the intricacies introduced by this localization, 
let us briefly consider the simpler case of $\vp = 1$ to better illustrate how information of the type found in (\ref{mi.1}) plays into showing that $\mathcal{E}_\vp$ stays bounded. In this case it is easy to see (cf. Lemma \ref{lem_E}) that 
\bas
  \mathcal{E}'_1(t) + \io \frac{|\nabla u|^2}{u} + \frac{1}{2} \io |\Delta v|^2 \leq 5\io u^2.
\eas
We can then make use of a subtle consequence of the Sobolev inequality (cf. Lemma \ref{lem788}), which has been previously used in \cite{BHN} and \cite{NSS2000}, combined with a global counterpart to information of the type seen in (\ref{mi.1}) to estimate
\bas
  \mathcal{E}'_1(t) &\leq& c_1 \exp\left(c_1 \io u\ln( u + e) \right) + c_1\left(\io u_0\right)^2 \leq c_1 \exp\left(c_1\delta \ln\frac{T}{T-t} \right) + c_1\left(\io u_0\right)^2  \\
  &\leq& c_2(T-t)^{-\frac{1}{2}} + c_2 
\eas
under the assumption that $c_1\delta < \frac{1}{2}$. In the whole-domain case, this already completes the argument. Notably even in the localized version of this argument, this is the central point where our choice of $\delta_0$ is crucial and it is thus easy to see that $\delta_0$ only depends on certain Sobolev constants and thus the domain (cf. Lemma \ref{lem799}).
\abs
While most of this argument readily transfers to the localized case as most additional terms introduced by the presence of a cut-off function can be easily absorbed by the dissipative structure seen above (cf. Lemma \ref{lem9}), the localization still introduces some new critical terms. To handle the most problematic of these terms, i.e.\ $\io \vp^3 u\ln(u)\nabla v \cdot \nabla \vp$, we employ localized elliptic regularity theory (cf. Lemma \ref{lem32} and \ref{lem33}) to split it up into some terms we can easily handle using dissipation as well as a term of the form $c V^2(t)$, where $V(t) := 1 + \int_{\supp \vp} |\nabla v|^2$, which remains to be controlled. 
\abs
While some basic integrability information for $V(t)$ can be easily extracted from the second equation in (\ref{0}) as seen in Lemma \ref{lem3}, to gain sufficient information about $V^2(t)$ we employ another localized version of a well-known energy functional (cf. Lemma \ref{lem7}), namely
\be{F0}
  \Fp(t)
	:=\frac{1}{2} \io \vp^4 |\na v(\cdot,t)|^2
	+ \frac{1}{2} \io \vp^4 v^2(\cdot,t)
	- \io \vp^4 u(\cdot,t) v(\cdot,t)
	+ \io \vp^4 u(\cdot,t) \big(\ln u(\cdot,t)-1\big).
\ee
Famously, the functional $\Fp$ is not necessarily bounded from below but using a consequence of the Trudinger-Moser inequality (cf. Lemma \ref{lem778} and \ref{lem79}) it is at least easy to see that 
\[
  \Fp(t) \geq \frac{1}{4}\int_\Omega \vp^4 |\nabla v|^2 - c \int_{\supp \vp} u \ln (u+e) - c.
\]
Thus showing that $\Fp$ is bounded from above combined with the information from (\ref{mi.1}) yields the last piece of the puzzle, i.e. integrability information for $V^2(t)$. While it is of course well-known that $\Fp' \leq 0$ for $\vp \equiv 1$, localization again introduces some added complications in the form of additional growth terms, which cannot be fully coped with using the dissipative properties of $\Fp$ alone. Fortunately, the dissipative structures of $\mathcal{E}_\vp$ and $\Fp$ complement each other in such a way that boundedness of a combined functional can be ensured (cf. Lemma \ref{lem8}), which yields the necessary estimate for $V^2(t)$.
\mysection{Preliminaries}
\subsection{Interpolation inequalities}
A crucial role in our analysis will be played by appropriate interpolation. In this section, we therefore separately collect
some tools in this vein.\abs
The following consequence of the two-dimensional Moser-Trudinger inequality arises as a special case of a
more general result recorded in \cite[Lemma 2.2]{win_SIMA}.
\begin{lem}\label{lem778}
  There exists $\Gamma_1=\Gamma_1(\Om)>0$ such that if $\phi\in C^0(\bom)$ and $\psi\in C^1(\bom)$ are nonnegative with $\phi\not\equiv 0$,
  then
  \bea{778.1}
	\io \phi\psi
	&\le& \frac{1}{a} \io \phi \cdot \bigg\{ \ln \phi - \ln \left(\frac{1}{|\Om|} \io \phi \right)\bigg\}
	+ \Gamma_1 a \cdot \bigg\{ \io \phi \bigg\} \cdot \io |\na\psi|^2
	+ \Gamma_1 a \cdot \bigg\{ \io \phi\bigg\} \cdot \bigg\{ \io \psi\bigg\}^2 \nn\\
	& & + \frac{\Gamma_1}{a} \io \phi
	\qquad \mbox{for all } a>0.
  \eea
\end{lem}
We next rely on continuity of the embedding $W^{1,1}(\Om) \hra L^2(\Om)$ to derive the following variant of some Ehrling-type
interpolation inequalities originally stated in \cite[Lemma 3.2]{NSS2000} (cf.~also the precedent in \cite[p.~1200]{BHN}).
\begin{lem}\label{lem788}
  With some $\Gamma_2=\Gamma_2(\Om)>0$, any $\vp\in C^1(\bom)$ and $\psi\in C^1(\bom)$ fulfilling
  $0\le \vp\le 1$ and $\psi> 0$ in $\bom$ have the property that
  \bea{788.1}
	\io \vp^4 \psi^2
	&\le& \eta \io \vp^4 \frac{|\na\psi|^2}{\psi}
	+ 3|\Om| \exp \bigg\{ \frac{\Gamma_2}{\eta} \cdot \int_{\supp\vp} \psi\ln(\psi+e) \bigg\} \nn\\
	& & + \Gamma_2 \left\{\|\na\vp\|_{L^\infty(\Om)}^2 + 1\right\} \cdot\bigg\{ \io \psi \bigg\}^2
	\qquad \mbox{for all } \eta>0.
  \eea
\end{lem}
\proof
  The Sobolev inequality in the two-dimensional domain $\Om$ provides $c_1=c_1(\Om)>0$ such that
  \be{788.2}
	\io \rho^2 \le c_1\cdot\bigg\{ \io |\na\rho|\bigg\}^2 + c_1 \cdot\bigg\{ \io |\rho|\bigg\}^2
	\qquad \mbox{for all } \rho\in W^{1,1}(\Om),
  \ee
  and given $\vp\in C^1(\bom;[0,1]), \psi\in C^1(\bom;(0,\infty))$ and $\eta>0$ we let $I:=\int_{\supp\vp} \psi\ln (\psi+e)$ and 
  $\xi:=e^{\frac{4c_1}{\eta}I}$, noting that $I> 0$ and hence $\ln\xi>0$.
  From (\ref{788.2}) we therefore obtain that
  \bas
	\io \vp^4 \psi^2
	&=& \int_{\{\psi>\xi\}} \vp^4 (\psi-\xi+\xi)^2	
	+ \int_{\{ \psi\le\xi\}} \vp^4 \psi^2 \\
	&\le& 2\io \vp^4 (\psi-\xi)_+^2
	+ 3|\Om| \xi^2 \\
	&\le& 2c_1 \cdot \bigg\{ \io \Big| \na \big( \vp^2 (\psi-\xi)_+\big) \Big| \bigg\}^2
	+ 2c_1 \cdot \bigg\{ \io \vp^2 (\psi-\xi)_+ \bigg\}^2
	+ 3|\Om| \xi^2 \\
	&\le& 4c_1 \cdot \bigg\{ \int_{\{\psi>\xi\}} \vp^2 |\na\psi| \bigg\}^2
	+ 16 c_1 \cdot \bigg\{ \io \vp (\psi-\xi)_+ |\na\vp| \bigg\}^2 \\
	& & + 2c_1 \cdot \bigg\{ \io \vp^2 (\psi-\xi)_+ \bigg\}^2
	+ 3|\Om|\xi^2 \\
	&\le& 4c_1 \cdot \bigg\{ \int_{\{\psi>\xi\}\cap\supp\vp} \psi \bigg\} \cdot \io \vp^4 \frac{|\na\psi|^2}{\psi}
	+ 16c_1 \|\na\vp\|_{L^\infty(\Om)}^2 \cdot \bigg\{ \io \psi\bigg\}^2 \\
	& & + 2c_1 \cdot\bigg\{ \io \psi \bigg\}^2
	+ 3|\Om|\xi^2 \\
	&\le& \frac{4c_1}{\ln\xi} \cdot \bigg\{ \int_{\{\psi>\xi\}\cap\supp\vp} \psi\ln\psi\bigg\} \cdot \io \vp^4 \frac{|\na\psi|^2}{\psi}
	+ 16c_1 \cdot \big\{ \|\na\vp\|_{L^\infty(\Om)}^2 + 1 \big\} \cdot \bigg\{ \io \psi\bigg\}^2 + 3|\Om|\xi^2.
  \eas
  Since
  \bas
	\frac{4c_1}{\ln\xi} \cdot \bigg\{ \int_{\{\psi>\xi\}\cap\supp\vp} \psi\ln\psi\bigg\} 
	\le \frac{4c_1}{\ln\xi} \cdot I = \eta
  \eas
  and
  \bas
	3|\Om|\xi^2
	= 3|\Om| \exp \bigg\{ \frac{8c_1}{\eta} \int_{\supp\vp} \psi\ln (\psi+e)\bigg\}
  \eas
  according to our definitions of $I$ and $\xi$, this entails (\ref{788.1}) with 
  $\Gamma_2\equiv \Gamma_2(\Om):=16c_1$.
\qed
In preparation for a second descendant of the Sobolev inequality, let us localize a consequence of elliptic $W^{2,2}$-regularity
as follows.
\begin{lem}\label{lem32}
  There exists $\Gamma_3=\Gamma_3(\Om)>0$ such that whenever $\vp\in C^2(\bom)$ and $\psi\in C^2(\bom)$ are such that
  $\frac{\pa\vp}{\pa\nu}=\frac{\pa\psi}{\pa\nu}=0$ on $\pO$ and $0\le\vp\le 1$,
  \bea{32.1}
	\io \vp^4 |D^2 \psi|^2
	&\le& \Gamma_3 \io \vp^4 |\Del \psi|^2 \nn \\
	& & + \Gamma_3 \cdot \big( \|D^2 \vp\|_{L^\infty(\Om)}^2 + \|\na\vp\|_{L^\infty(\Om)}^4 + 1\big) \cdot
		\bigg\{ \int_{\supp\vp} |\na\psi|^2 + \int_{\supp\vp} \psi^2  \bigg\}.
  \eea
\end{lem}
\proof
  In line with standard elliptic regularity theory (\cite{GT}), we can find $c_1=c_1(\Om)>0$ such that for all $(i,j)\in\{1,2\}^2$ we have
  \be{32.2}
	\io |\pa_{x_i} \pa_{x_j} \rho|^2
	\le c_1 \io |\Del\rho|^2 + c_1 \io \rho^2
	\qquad \mbox{for all $\rho\in C^2(\bom)$ fulfilling $\frac{\pa\rho}{\pa\nu}=0$ on } \pO.
  \ee
  Given $\vp$ an $\psi$ with the listed properties, letting $\rho:=\vp^2\psi$ and computing
  \bas
	\pa_{x_i} \pa_{x_j} \rho
	= \vp^2 \pa_{x_i} \pa_{x_j} \psi
	+ 2\vp \pa_{x_i} \vp \pa_{x_j} \psi
	+ 2\vp \pa_{x_j}\vp \pa_{x_i} \psi
	+ 2\vp \pa_{x_i} \pa_{x_j} \vp \psi
	+ 2\pa_{x_i} \vp \pa_{x_j} \vp \psi
  \eas
  we readily see that
  \be{32.3}
	|\pa_{x_i} \pa_{x_j} \rho|^2
	\ge \frac{1}{2} \vp^4 |\pa_{x_i} \pa_{x_j} \psi|^2
	- 32 \vp^2 |\na\vp|^2 |\na\psi|^2
	- 16 \vp^2 |D^2\vp|^2 \psi^2
	- 16 |\na\vp|^4 \psi^2
	\qquad \mbox{in } \Om
  \ee
  for $(i,j)\in\{1,2\}^2$, because $(A+B)^2 \ge \frac{1}{2} A^2 - B^2$ and $(A+B+C+D)^2 \le 4(A^2+B^2+C^2+D^2)$ for 
  $(A,B,C,D)\in\R^4$.
  We furthermore, similarly, estimate
  \bea{32.4}
	|\Del\rho|^2
	&=& \big| \vp^2 \Del\psi + 4\vp\na\vp\cdot\na\psi + 2\vp \Del\vp \cdot \psi + 2|\na\vp|^2 \psi\big|^2 \nn \\
	&\le& 4\vp^4 |\Del\psi|^2
	+ 64 \vp^2 |\na\vp|^2 |\na\psi|^2
	+ 16 \vp^2 |\Del\vp|^2 \psi^2
	+ 16 |\na\vp|^4 \psi^2.
  \eea
  Then by first applying (\ref{32.3}) followed by (\ref{32.2}), 
  applicable here since $\frac{\pa\rho}{\pa\nu}=\vp^2 \frac{\pa\psi}{\pa\nu}+ 2\vp\psi \frac{\pa\vp}{\nu}=0$ on $\pO$, and finally (\ref{32.4}),
  we obtain that
  \bas
	\frac{1}{2} \io \vp^4 |\pa_{x_i} \pa_{x_j} \psi|^2
	&\le& 32 \|\vp\|_{L^\infty(\Om)}^2 \|\na\vp\|_{L^\infty(\Om)}^2 \int_{\supp\vp} |\na\psi|^2
	+ 16\|\vp\|_{L^\infty(\Om)}^2 \|D^2\vp\|_{L^\infty(\Om)}^2 \int_{\supp\vp} \psi^2 \\
	& & + 16 \|\na\vp\|_{L^\infty(\Om)}^4 \int_{\supp\vp} \psi^2 + c_1 \io |\Del \rho|^2 + c_1 \io \rho^2 \\
  &\le& 32 \|\vp\|_{L^\infty(\Om)}^2 \|\na\vp\|_{L^\infty(\Om)}^2 \int_{\supp\vp} |\na\psi|^2
	+ 16\|\vp\|_{L^\infty(\Om)}^2 \|D^2\vp\|_{L^\infty(\Om)}^2 \int_{\supp\vp} \psi^2 \\
  & & + 16 \|\na\vp\|_{L^\infty(\Om)}^4 \int_{\supp\vp} \psi^2 \\
	& & + 4 c_1 \io \vp^4 |\Del\psi|^2 
	+ 64 c_1 \|\vp\|_{L^\infty(\Om)}^2 \|\na\vp\|_{L^\infty(\Om)}^2 \int_{\supp\vp} |\na\psi|^2 \\
	& & + 16 c_1 \|\vp\|_{L^\infty(\Om)}^2 \|\Del\vp\|_{L^\infty(\Om)}^2 \int_{\supp\vp} \psi^2
	+ 16 c_1 \|\na\vp\|_{L^\infty(\Om)}^4 \int_{\supp\vp} \psi^2 \\
  & & + c_1 \|\vp\|_{L^\infty(\Om)}^4 \int_{\supp\vp} \psi^2,
  \eas
  which readily implies (\ref{32.1}) by grouping the integral terms containing $\psi$ and recalling that our assumptions ensure that $\|\vp\|_{L^\infty(\Om)} \leq 1$.
\qed
We can thereby establish the following localized variant of an essentially well-known Gagliardo--Nirenberg type inequality.
\begin{lem}\label{lem33}
  There exists $\Gamma_4=\Gamma_4(\Om)>0$ with the property that if $\vp\in C^2(\bom)$ and $\psi\in C^2(\bom)$ satisfy
  $\frac{\pa\vp}{\pa\nu}=\frac{\pa\psi}{\pa\nu}=0$ on $\pO$ and $0\le\vp\le 1$, then
  \bea{33.1}
	\io \vp^4 |\na\psi|^4
	&\le& \Gamma_4 \cdot \bigg\{ \int_{\supp\vp} |\na\psi|^2 \bigg\} \cdot \io \vp^4 |\Del\psi|^2 \nn\\
	& & \hs{-16mm}
	+ \Gamma_4 \cdot \big( \|D^2 \vp\|_{L^\infty(\Om)}^2 + \|\na\vp\|_{L^\infty(\Om)}^4 +1 \big) \cdot
		\Bigg\{ \bigg\{ \int_{\supp\vp} |\na\psi|^2 \bigg\}^2 + \bigg\{ \int_{\supp\vp} \psi^2 \bigg\}^2 \Bigg\}.
  \eea
\end{lem}
\proof
  We again fix $c_1=c_1(\Om)>0$ such that in line with the Sobolev inequality,
  \bas
	\|\rho\|_{L^2(\Om)}^2 \le c_1\|\na\rho\|_{L^1(\Om)}^2 + c_1\|\rho\|_{L^1(\Om)}^2
	\qquad \mbox{for all } \rho\in W^{1,1}(\Om).
  \eas
  Abbreviating $J_1:=\int_{\supp\vp} |\na\psi|^2$ we therefore see that
  \bea{33.2}
	\io \vp^4 |\na\psi|^4
	&=& \Big\| \vp^2 |\na\psi|^2 \Big\|_{L^2(\Om)}^2 \nn\\
	&\le& c_1 \cdot \bigg\{ \io \vp^2 \Big| \na |\na\psi|^2 \Big| + 2\io \vp |\na\psi|^2 |\na\vp| \bigg\}^2 
	+ c_1\cdot \bigg\{ \io \vp^2 |\na\psi|^2 \bigg\}^2 \nn\\
	&\le& 2c_1 \cdot \bigg\{ \io \vp^2 \Big| \na |\na\psi|^2 \Big| \bigg\}^2
	+ 8c_1 \|\na\vp\|_{L^\infty(\Om)}^2 \cdot J_1^2 + c_1 J_1^2,
  \eea
  because $0\le \vp\le 1$.
  Since $\na |\na\psi|^2 = 2D^2 \psi\cdot\na\psi$, in view of the Cauchy-Schwarz inequality and Lemma \ref{lem32} we furthermore
  obtain that with $\Gamma_3=\Gamma_3(\Om)$ found there, and with $J_2:=\int_{\supp\vp} \psi^2$,
  \bas
	& & 2c_1 \cdot\bigg\{ \io \vp^2 \Big|\na |\na\psi|^2 \Big| \bigg\}^2 \\
	&\le& 8c_1 \cdot \bigg\{ \io \vp^2 |\na\psi| \cdot |D^2\psi| \bigg\}^2 \\
	&\le& 8c_1 J_1 \io \vp^4 |D^2 \psi|^2 \\
	&\le& 8c_1 \Gamma_3 J_1 \io \vp^4 |\Del\psi|^2
	+ 8c_1 \Gamma_3 \cdot \big( \|D^2\vp\|_{L^\infty(\Om)}^2 + \|\na\vp\|_{L^\infty(\Om)}^4 + 1 \big) \cdot J_1 (J_1+J_2).
  \eas
  Estimating $J_1(J_1+J_2) \le 2(J_1^2+J_2^2)$ by Young's inequality here, from (\ref{33.2}) we obtain (\ref{33.1}). 
\qed

\subsection{A family of cut-off functions}
To complement the localized interpolation inequalities of the previous section, we will now briefly present a construction of compatible cut-off functions based on some standard results from differential geometry. Most crucially, our family of cut-off functions will feature the boundary conditions necessary for the elliptic regularity theory necessary for some of the aforementioned functional inequalities.
\abs
Constructions of this type have naturally already been discussed in similar contexts in e.g.\ \cite{mizoguchi_souplet} and \cite{NSS2000}. But given the subtle differences in the properties ensured by those constructions to what we need here and given how central these families of cut-off functions are to our argument, we will still present our specific construction in full here. As a convenient feature of the techniques from the theory of manifolds at the core of our argument, our construction should further be easy to adapt to future needs in similar contexts and yields many desirable properties such as uniform boundedness (after multiplication with the scaling parameter) as an inherent feature of the methods involved without any further effort.
\begin{lem}\label{lem21}
  There exist $R_0 > 0$, $\theta \in (0,1)$ and $\Lam > 0$ such that for any choice of $x_0\in \R^2$ and $R \in (0,R_0)$  one can find
  $\vp=\vp^{(x_0,R)}\in C^2(\R^2)$ such that $0\le \vp\le 1$ in $\Om$ and
  \be{21.1}
	\vp\equiv 1	
	\quad \mbox{in } B_{\theta R}(x_0),
  \ee
  that
  \be{21.2}
	\supp \vp \subset B_R(x_0),
  \ee
  that
  \be{21.3}
	\frac{\pa\vp}{\pa\nu}(x)=0
	\qquad \mbox{for all } x\in\pO,
  \ee
  and that
  \be{21.4}
	R |\na\vp(x)| + R^2 |D^2 \vp(x)| \le \Lam
	\qquad \mbox{for all } x\in\R^2.
  \ee
\end{lem}
\proof 
  As the standard characterizations of smooth boundary regularity of $\Omega$ in particular always entail that $\pO$ is an embedded smooth submanifold of $\R^2$ (\cite{wloka_pde_book}), we can employ standard results from differential geometry (\cite[Theorem 6.24]{lee_manifold_book}) to find a tubular neighborhood $\mathcal{T} = \{ x \in \R^2 \, | \, \dist(x, \pO) \leq \delta \}$ with some $\delta > 0$ and a smooth projection $P: \mathcal{T} \rightarrow \pO$ with 
  \be{21.5}
    P(x + t \nu(x)) = x \qquad \text{ for all } x \in \partial \Om, t \in [-\delta, \delta]
  \ee 
  as well as 
  \be{21.6}
    |x - P(x)| = \dist(x, \pO) \qquad \mbox{for all } x \in \mathcal{T},
  \ee 
  where $\nu(x)$ is the standard normalized outward normal vector at $x$. 
  Notably, our ability to find a minimum "thickness" $\delta$ for the whole tubular neighborhood $\mathcal{T}$ is due to the fact that the "thickness" of $\mathcal{T}$ at each point $x_0 \in \pO$ depends continuously on said point $x_0$ and our boundary is compact.
  \abs
  We now fix a standard cut-off function $\psi \in C^\infty(\R^2; [0,1])$ such that $\supp \psi \subset B_1(0)$ and $\psi \equiv 1$ on $B_\frac{1}{2}(0)$. We then set 
  \be{21.7}
    R_0 := \frac{\delta}{2} \quad \text{ and } \quad \theta := \min\left\{ \frac{1}{16}, \frac{1}{8\|D^1 P\|_{L^\infty(\mathcal{T})}}, \frac{1}{8\|D^1(I - P)\|_{L^{\infty}(\mathcal{T})}} \right\},
  \ee
  where $I:\R^2 \rightarrow \R^2$ is the identity map.
  Let now $x_0 \in \R^2$ and $R \in (0,R_0)$.
  \abs
  We first discuss the case of $\dist(x_0,\pO) < \frac{R}{8}$. We then set 
  \bas
    \vp(x) := \psi\left(\tfrac{4}{R}( P(x_0) - P(x) )\right)\cdot \psi\left(\tfrac{2}{R} ( x - P(x) )\right) \qquad \mbox{for all } x \in \mathcal{T}. 
  \eas 
  If now $x \in \mathcal{T}$ with $|x_0 - x| > R$, then 
  \bas
    R < |x_0 - x| &\leq & |x_0 - P(x_0)| + |P(x_0) - P(x)| + |P(x) - x| \\
    &<& \frac{R}{8} + |P(x_0) - P(x)| + |P(x) - x|  
  \eas
  due to (\ref{21.6}) and thus either $|P(x_0) - P(x)|$ must have been larger than $\frac{R}{4}$ or $|P(x) - x|$ must have already been larger than $\frac{5R}{8} > \frac{R}{2}$. This directly implies that $\supp \vp \subset B_{R}(x_0)$ and therefore $\vp$ can be smoothly extended to $\R^2$ by zero as $B_{R}(x_0) \subset B_{R_0}(x_0) \subset \mathcal{T}$ by our choice of $R_0$ in (\ref{21.7}). Conversely if $x \in T$ with $|x_0 - x| < \theta R$, then 
  \bas
    |P(x_0) - P(x)| \leq \|D^1 P\|_{L^\infty(T)} |x_0 - x| < \|D^1 P\|_{L^\infty(T)} \theta R \leq \frac{R}{8}.
  \eas
  and similarly
  \bas
    |x - P(x)| &\leq& |x - P(x) - (x_0 - P(x_0) )| + |x_0 - P(x_0)| \\ 
    &<& \|D^1 (I-P)\|_{L^\infty(\mathcal{T})}\theta R + |x_0 - P(x_0)| \leq \frac{R}{8} + \frac{R}{8} = \frac{R}{4} 
  \eas 
  by the mean value inequality, (\ref{21.6}) and (\ref{21.7}). This immediately gives us (\ref{21.2}).
  \abs
  Further due to (\ref{21.5}) and our choice of $\psi$ it follows that 
  \bas
    \frac{\vp(x + t \nu(x)) - \vp(x)}{t} = \frac{1}{t} \psi\left(\tfrac{4}{R}( P(x_0) - x )\right) \left\{ \psi\left(\tfrac{2t \nu(x)}{R}\right)  - \psi\left(0\right)  \right\} = 0
  \eas
  for all $x \in \partial \Omega$ and $t \in (-\frac{R}{4}, \frac{R}{4})$ and thus (\ref{21.3}).
  \abs
  We now discuss the case of $\dist(x_0,\pO) \ge \frac{R}{8}$. We then simply set 
  \bas
    \vp(x) := \psi\left( \tfrac{8}{R}( x_0 - x ) \right) \qquad \mbox{for all } x \in \R^2.
  \eas  
  As the support of $\vp$ in this case is fully contained in $B_{R/8}(x_0)$ and thus does not intersect $\pO$ by construction and $\theta \leq \frac{1}{16}$ by definition, (\ref{21.1}), (\ref{21.2}) and (\ref{21.3}) are straightforward to verify.
  \abs
  Finally as in both of our constructions $R$ only features as the scaling factor $\frac{1}{R}$ and the only other ingredients are the smooth and compactly supported functions $\psi$ and $P$, it is easy to find a uniform constant $\Lam > 0$ only depending on the finite norms $\|P\|_{C^2(\mathcal{T})}$ and $\|\psi\|_{C^2(\R^2)}$ such that (\ref{21.4}) holds by straightforward applications of the chain and product rules. 
\qed
\subsection{Basic regularity properties}
Let us now launch our analysis of the evolution system (\ref{0}) by stating, for convenient later reference, 
two simple properties of the solutions obtained in Proposition \ref{prop0}. 
Firstly, mass conservation immediately results from an integration:
\begin{lem}\label{lem_mass}
  If (\ref{init}) holds, then
  \be{mass}
	\io u(\cdot,t)=\io u_0
	\qquad \mbox{for all } t\in (0,T).
  \ee
\end{lem}
In the considered planar setting, it is well-known that
(\ref{mass}) implies the following basic integrability feature of the second solution component.
\begin{lem}\label{lem2}
  Assume (\ref{init}). Then for all $p\in [1,\infty)$, there exists $C=C(p,u_0,v_0)>0$ such that
  \be{2.1}
	\|v(\cdot,t)\|_{L^p(\Om)} \le C
	\qquad \mbox{for all } t\in (0,T).
  \ee
\end{lem}
\proof
  This can be obtained in a straightforward manner from (\ref{mass}) and appropriate heat semigroup estimates;
  for instance, (\ref{2.1}) directly follows from (\ref{reg}) and \cite[Lemma 4.1]{horstmann_win}.
\qed
\subsection{Evolution properties of the energy-type functionals $\Ep$ and $\Fp$}
As a final preparation, we will now derive some fundamental evolution properties of the two localized energy-type functionals $\Ep$ and $\Fp$ by way of several straightforward testing procedures. These two functionals will form the core of our later analysis of the system (\ref{0}) by not only carefully studying them individually but also by combining them to exploit mutually beneficially dissipative structures. \abs
As the quantity $\io \vp^4 u(\ln u -1)$ not only occurs in both these functionals but will also be treated in essentially the same way in both cases (as opposed to the quantity $\io \vp^4 |\na v|^2$ that is also present in both functionals but is treated very differently here), we begin by deriving the following quick lemma as a matter of convenience.

\begin{lem}\label{lem6}
  Let $\vp\in C^1(\bom)$. Then
  \be{6.1}
	\frac{d}{dt} \io \vp^4 u(\ln u -1) + \io \vp^4 \frac{|\na u|^2}{u}
	=
	\io \vp^4 \na u\cdot\na v
	- 4 \io \vp^3 \ln u\cdot (\na u-u\na v) \cdot \na\vp
  \ee
  for all $t\in (0,T)$.
\end{lem}
\proof
  We use the first equation in (\ref{0}) to compute 
  \bas
	\frac{d}{dt} \io \vp^4 u(\ln u-1)
	&=& \io \vp^4 \ln u\cdot u_t \\
	&=& \io \vp^4 \ln u\na\cdot (\na u-u\na v) \\
	&=& -\io \Big\{ \vp^4 \frac{\na u}{u} + 4\vp^3 \ln u\na \vp\Big\} \cdot (\na u-u\na v)
	\qquad \mbox{for all } t\in (0,T),
  \eas
  which is equivalent to (\ref{6.1}).
\qed
We can now control the growth of the functional in (\ref{E}):
\begin{lem}\label{lem_E}
  Given $\vp\in C^1(\bom)$, let
  \be{E}
	\Ep(t)
	:= \io \vp^4 u(\cdot,t) \big(\ln u(\cdot,t)-1\big) + \io \vp^4 |\na v(\cdot,t)|^2
  \ee
  for all $t \in [0,T)$. Then 
  \bea{E_evol}
    \Ep'(t) + \io \vp^4 \frac{|\na u|^2}{u} + \io \vp^4 |\Delta v|^2 &\le& \io\vp^4 \na u \cdot \na v - 4\io\vp^3 \ln u \cdot (\na u - u \na v)\cdot \na \vp \nn\\
    & &+ 4 \io \vp^4 u^2 + 40 \io \vp^2 |\na v|^2 |\na\vp|^2
  \eea
  for all $t \in (0,T)$.
\end{lem}
\proof
  From the second equation in (\ref{0}) we directly find that
  \bea{E_1}
	\frac{1}{2} \frac{d}{dt} \io \vp^4 |\na v|^2
	&=& \io \vp^4 \na v\cdot\na v_t \nn \\
	&=& \io \vp^4 \na v\cdot\na (\Del v+u)
	- \io \vp^4 |\na v|^2 \nn \\
	&=& - \io \vp^4 |\Del v|^2
	- \io \vp^4 u\Del v
	- 4 \io \vp^3 \Del v \na v\cdot\na \vp
	- 4 \io \vp^3 u\na v\cdot\na\vp \nn \\
	& & - 4 \io \vp^4 |\na v|^2
	\qquad \mbox{for all } t\in (0,T).
  \eea 
  By using Young's inequality to estimate
  \bas
	- \io \vp^4 u\Del v
	\le \frac{1}{4} \io \vp^4 |\Del v|^2
	+ \io \vp^4 u^2
  \eas
  and 
  \bas
	- 4 \io \vp^3 \Del v \na v\cdot\na\vp
	\le \frac{1}{4} \io \vp^4 |\Del v|^2
	+ 16 \io \vp^2 |\na v|^2 |\na\vp|^2
  \eas
  as well as
  \bas
	- 4 \io \vp^3 u \na v\cdot\na\vp
	\le \io \vp^4 u^2 
	+ 4 \io \vp^2 |\na v|^2 |\na\vp|^2
  \eas
  for all $t\in (0,T)$,
  from (\ref{E_1}) we obtain after dropping a favorably signed summand that
  \bas
  \frac{d}{dt} \io \vp^4 |\na v|^2
	+ \io \vp^4 |\Del v|^2
	\le 4 \io \vp^4 u^2
	+ 40 \io \vp^2 |\na v|^2 |\na\vp|^2
	\qquad \mbox{for all } t\in (0,T).
  \eas
  Combining this estimate with the one from Lemma \ref{lem6} then completes the proof.
\qed
Some basic evolution properties of the expressions in (\ref{F}) can be obtained in a comparably simple manner:
\begin{lem}\label{lem7}
  Given $\vp\in C^1(\bom)$, let
  \be{F}
	\Fp(t)
	:=\frac{1}{2} \io \vp^4 |\na v(\cdot,t)|^2
	+ \frac{1}{2} \io \vp^4 v^2(\cdot,t)
	- \io \vp^4 u(\cdot,t) v(\cdot,t)
	+ \io \vp^4 u(\cdot,t) \big(\ln u(\cdot,t)-1\big)
  \ee
  for all $t \in [0,T)$.
  Then 
  \bea{7.1}
	\Fp'(t)
	+ \frac{1}{2} \io \vp^4 \Big| \frac{\na u}{\sqrt{u}}-\sqrt{u} \na v\Big|^2
	&\le& 16 \io \vp^2 u\ln^2 u |\na\vp|^2
	+ 16 \io \vp^2 uv^2 |\na\vp|^2 \nn\\
	& & + 4 \io \vp^2 |\na v|^2 |\na\vp|^2
  \eea
  for all $t \in (0,T)$.
\end{lem}
\proof
  According to (\ref{0}),
  \bas
	\frac{1}{2} \frac{d}{dt} \io \vp^4 |\na v|^2
	= \io \vp^4 \na v\cdot\na v_t 
	= - \io \vp^4 \Del v \cdot v_t 
	- 4 \io \vp^3 v_t \na v\cdot\na\vp
  \eas
  and
  \bas
	\frac{1}{2} \frac{d}{dt} \io \vp^4 v^2 
	= \io \vp^4 vv_t
  \eas
  as well as
  \bas
	- \frac{d}{dt} \io \vp^4 uv
	&=& - \io \vp^4 uv_t
	- \io \vp^4 v\na\cdot (\na u-u\na v) \\
	&=& - \io \vp^4 uv_t
	+ \io \vp^4 \na u\cdot \na v	
	- \io \vp^4 u |\na v|^2 
	+ 4 \io \vp^3 v (\na u-u\na v) \cdot\na\vp
  \eas
  for all $t\in (0,T)$, whence using Lemma \ref{lem6} we see that
  \bas
	\Fp'(t)
	&=& - \io \vp^4 (\Del v - v + u) v_t
	- \io \vp^4 \Big\{ \frac{|\na u|^2}{u} - 2\na u\cdot\na v + u|\na v|^2 \Big\} \nn\\
	& & - 4 \io \vp^3 v_t \na v\cdot\na\vp
	+ 4\io \vp^3 v(\na u-u\na v)\cdot\na\vp \nn\\
	& & - 4 \io \vp^3 \ln u (\na u-u\na v) \cdot\na\vp
	\qquad \mbox{for all } t\in (0,T).
  \eas
  Since $\Del v - v + u = v_t$ and $\frac{|\na u|^2}{u} - 2\na u\cdot\na v + u|\na v|^2= \big| \frac{\na u}{\sqrt{u}}-\sqrt{u}\na v\big|^2$
  as well as $\na u-u\na v=\sqrt{u} \big\{ \frac{\na u}{\sqrt{u}}-\sqrt{u}\na v\big\}$ in $\Om\times (0,T)$, this means that
  \bea{7.3}
	\Fp'(t)
	+ \io \vp^4 v_t^2
	+ \io \vp^4 \Big| \frac{\na u}{\sqrt{u}}-\sqrt{u}\na v\Big|^2
	&=& - 4 \io \vp^3 v_t \na v\cdot\na\vp
	+ 4\io \vp^3 \sqrt{u} v \Big\{ \frac{\na u}{\sqrt{u}}-\sqrt{u}\na v\Big\} \cdot\na \vp \nn\\
	& & - 4 \io \vp^3 \sqrt{u}\ln u \Big\{ \frac{\na u}{\sqrt{u}}-\sqrt{u}\na v\Big\} \cdot\na\vp 
%	\qquad \mbox{for all } t\in (0,T).
  \eea
  for all $t\in (0,T)$.
  Here, Young's inequality ensures that
  \bas
	- 4 \io \vp^3 v_t \na v\cdot\na\vp
	\le \io \vp^4 v_t^2
	+ 4 \io \vp^2 |\na v|^2 |\na\vp|^2
	\qquad \mbox{for all } t\in (0,T),
  \eas
  that
  \bas	
	4 \io \vp^3 \sqrt{u} v\Big\{ \frac{\na u}{\sqrt{u}}-\sqrt{u}\na v\Big\} \cdot\na\vp
	\le \frac{1}{4} \io \vp^4 \Big|\frac{\na u}{\sqrt{u}}-\sqrt{u}\na v\Big|^2
	+ 16 \io \vp^2 uv^2 |\na\vp|^2
	\qquad \mbox{for all } t\in (0,T),
  \eas
  and that
  \bas
	-4 \io \vp^3 \sqrt{u} \ln u \Big\{ \frac{\na u}{\sqrt{u}}-\sqrt{u}\na v\Big\}\cdot\na\vp
	\le \frac{1}{4} \io \vp^4 \Big| \frac{\na u}{\sqrt{u}}-\sqrt{u}\na v\Big|^2
	+ 16 \io \vp^2 u\ln^2 u |\na\vp|^2
%	\qquad \mbox{for all } t\in (0,T),
  \eas
  for all $t\in (0,T)$, so that (\ref{7.1}) results from (\ref{7.3}).
\qed
\mysection{Time-independent local bounds for $u$ in $L\log L$. Proof of Theorem \ref{theo11}}
Throughout this section we shall suppose that $(u_0,v_0)$ satisfies (\ref{init}) and is such that in Proposition \ref{prop0} we have $T<\infty$.
In order to derive Theorem \ref{theo11} by means of a contradiction-based argument, we then assume that for some $x_0 \in \B$ there further exists a time $t_0 \in [0,T)$, radius $R>0$ and $\del\in (0,1)$ such that
\be{H}
	\int_{B_R(x_0)\cap\Om} u(\cdot,t) \ln \big( u(\cdot,t)+e\big) \le \del \cdot \ln \frac{T}{T-t}
	\qquad \mbox{for all } t\in (t_0,T),
	\tag{H}
\ee
where without loss of generality we may assume that $R<R_0$ with $R_0 > 0$ as in Lemma \ref{lem21}.
\subsection{A localized space-time integrability property of $\na v$}
Let us first make sure that assuming (\ref{H}) entails space-time integrability of the signal gradient in domains slightly smaller
than those in (\ref{H}).
We remark that, although we utilize the cut-off functions from Lemma \ref{lem21} here for convenience, we do not make use of the full strength of their construction as we do not rely on the 
boundary information property in (\ref{21.3}) here just yet.
\begin{lem}\label{lem3}
  Let $u_0$ and $v_0$ be such that (\ref{init}) holds and $T<\infty$. Let further $R_0$ and $\theta$ be as in Lemma \ref{lem21} and assume that (\ref{H}) be valid for some blow-up point $x\in\B$, time $t_0 \in [0,T)$, radius $R \in (0,R_0)$ and $\del\in (0,1)$. 
  Then there exists $C=C(x_0,t_0,R,u_0,v_0)>0$ such that
  \be{3.1}
	\int_{t_0}^t \int_{B_{\theta R}(x_0) \cap \Om} |\na v|^2 \le C
	\qquad \mbox{for all } t\in (t_0,T).
  \ee
\end{lem}
\proof
  We let $\vp=\vp^{(x_0,R)}$ and $\Lam$ be as in Lemma \ref{lem21}. Testing the second equation in (\ref{0}) with $\vp^2 v$ then yields that
  \bas 
    \frac{1}{2}\frac{d}{dt} \io \vp^2 v^2 & = & \io \vp^2 v \Delta v  -\io \vp^2 v^2 + \io \vp^2 u v \\
    & = & -\io \vp^2 |\na v|^2 - 2\io \vp v (\na \vp \cdot \na v) -\io \vp^2 v^2 + \io \vp^2 u v
  \eas
  and thus that
  \be{3.2}
	\frac{d}{dt} \io \vp^2 v^2 + \io \vp^2 |\na v|^2
	\le 4c_1^2 \io v^2 + 2 \io \vp^2 uv
	\qquad \mbox{for all } t\in (0,T)
  \ee
  with $c_1\equiv c_1(R):=\frac{\Lam}{R}$ after a straightforward application of Young's inequality as well as (\ref{21.4}) and dropping of a favorably signed summand.
  Here, letting $\Gamma_1=\Gamma_1(\Om)$ be as in Lemma \ref{lem778} and taking $a=a(u_0)>0$ small enough such that with 
  $m:=\io u_0$ we have $4\Gamma_1 am \le \frac{1}{2}$, in line with (\ref{778.1}), (\ref{21.4}), (\ref{H}) and the fact that $-\xi\ln\xi\le\frac{1}{e}$
  for all $\xi>0$ we can estimate
  \bas
	2 \io \vp^2 uv
	&=& 2 \io (\vp u)\cdot (\vp v) \\
	&\le& \frac{2}{a} \io \vp u\ln(\vp u)
	- \frac{2|\Om|}{a} \cdot \bigg\{ \frac{1}{|\Om|} \io \vp u \bigg\} \cdot \ln \bigg\{ \frac{1}{|\Om|} \io \vp u \bigg\} \\
	& & + 2 \Gamma_1 a \cdot \bigg\{ \io \vp u\bigg\} \cdot \io \big| \vp \na v + v\na\vp\big|^2 \\
	& & + 2 \Gamma_1 a \cdot \bigg\{ \io \vp u\bigg\} \cdot \bigg\{ \io \vp v\bigg\}^2 
	+ \frac{2\Gamma_1}{a} \io \vp u \\
  &\le& \frac{2}{a} \int_{\supp \vp} \vp u\ln (\vp u+e)
	+ \frac{2|\Om|}{ae}  \\
	& & + 4\Gamma_1 am \io \vp^2 |\na v|^2 
  + 4\Gamma_1 am \io v^2 |\na\vp|^2  
  + 2 \Gamma_1 am \cdot \bigg\{ \io \vp v \bigg\}^2 
	+ \frac{2\Gamma_1 m}{a}  \\
	&\le& \frac{2}{a} \ln \frac{T}{T-t} 
	+ \frac{2|\Om|}{ae}
	+ \frac{1}{2} \io \vp^2 |\na v|^2
	+ \frac{c_1^2}{2}\io v^2  
  + \frac{1}{4} \bigg\{ \io v \bigg\}^2 
	+ \frac{2\Gamma_1 m}{a}
  \eas
  for all $t \in (t_0,T)$ because $0\le\vp\le 1$, $\supp\vp\subset B_R(x_0)$ and $\del\in (0,1)$.
  Using that Lemma \ref{lem2} provides $c_2=c_2(u_0,v_0)>0$ such that $\io v^2 + \io v \le c_2$ for all $t\in (0,T)$,
  from (\ref{3.2}) we thus infer that
  \bas
	\frac{d}{dt} \io \vp^2 v^2
	+ \frac{1}{2} \io \vp^2 |\na v|^2
	\le \frac{2}{a} \ln\frac{T}{T-t} + c_3
	\qquad \mbox{for all } t\in (t_0,T)
  \eas
  with $c_3\equiv c_3(R,u_0,v_0):= \frac{2|\Om|}{ae} + \frac{c_1^2 c_2}{2} + \frac{c_2^2}{4}  + \frac{2\Gamma_1 m}{a}$.
  Since $\int_{t_0}^T \ln\frac{T}{T-t} dt<\infty$, time integration yields the claim.
\qed
\subsection{Mildly time-dependent localized bounds for $\int |\na v|^2 dx$ by tracing an augmented energy functional}
In this key section, we intend to turn Lemma \ref{lem3} into an estimate for spatial integrals of $|\na v|^2$ in yet smaller domains,
involving a logarithmic and hence mild dependence on time.
This will be achieved in Lemma \ref{lem8} by an analysis of the functional $\Fp$ augmented by the dissipative structure of $\Ep$,
and in order 
to foreshadow that Lyapunov-like properties of the localized energy functional from (\ref{F}) indeed will be expedient for
our purposes in this regard, let us first derive a lower bound for $\Fp$ which in the presence of the hypothesis in (\ref{H})
exhibits a conveniently controllable dependence on the time variable near the instant of blow-up.
\begin{lem}\label{lem79}
  There exists $\Gamma_5=\Gamma_5(u_0,v_0)>0$ such that for any $\vp\in C^1(\bom)$ fulfilling $0\le\vp\le 1$, the function $\Fp$ defined
  in (\ref{F}) satisfies
  \bea{79.1}
	\Fp(t) \ge - \Gamma_5 \int_{\supp\vp} u\ln (u+e) - \Gamma_5 \|\na\vp\|_{L^\infty(\Om)}^2 - \Gamma_5
	\qquad \mbox{for all } t\in (0,T).
  \eea
\end{lem}
\proof
  We take $\Gamma_1=\Gamma_1(\Om)$ from Lemma \ref{lem778}, and fix $a=a(u_0)>0$ small enough such that writing $m:=\io u_0$ we have
  $2\Gamma_1 ma \le \frac{1}{2}$.
  Then whenever $\vp\in C^1(\bom)$ satisfies $0\le\vp\le 1$, using (\ref{mass}) and the fact that
  \bas
	\io \big|\na(\vp^2 v)\big|^2
	= \io | \vp^2\na v + 2\vp v\na\vp|^2
	\le 2\io \vp^4 |\na v|^2
	+ 8 \io \vp^2 v^2 |\na\vp|^2
	\qquad \mbox{for all } t\in (0,T)
  \eas 
  we infer from (\ref{778.1}) by an argument similar to that in Lemma \ref{lem3} that
  \bea{79.2}
	\io \vp^4 uv
	&=& \io (\vp^2 u)\cdot (\vp^2 v) \nn\\
	&\le& \frac{1}{a} \io \vp^2 u \cdot \ln (\vp^2 u)
	+ \frac{|\Om|}{ae}
	+ \Gamma_1 a \cdot \bigg\{ \io \vp^2 u\bigg\} \cdot \io \big| \na (\vp^2 v)\big|^2 \nn\\
	& & + \Gamma_1 a \cdot \bigg\{ \io \vp^2 u\bigg\} \cdot \bigg\{ \io \vp^2 v\bigg\}^2
	+ \frac{\Gamma_1}{a} \cdot \io \vp^2 u \nn\\
	&\le& \frac{1}{a} \int_{\supp\vp} \vp^2 u\ln (\vp^2 u + e)
	+ \frac{|\Om|}{ae}
	+ \Gamma_1 a \cdot \bigg\{ \io u\bigg\} \cdot \bigg\{ 2 \io \vp^4 |\na v|^2 + 8 \io \vp^2 v^2 |\na\vp|^2 \bigg\} \nn\\
	& & + \Gamma_1 a \cdot \bigg\{ \io u \bigg\} \cdot \bigg\{ \io v\bigg\}^2 
	+ \frac{\Gamma_1}{a} \io u \nn\\
	&\le& \frac{1}{a} \int_{\supp\vp} u\ln (u+e)
	+ \frac{|\Om|}{ae}
	+ \frac{1}{2} \io \vp^4 |\na v|^2
	+ 2 \|\na\vp\|_{L^\infty(\Om)}^2 \io v^2 \nn\\
	& & + \frac{1}{4} \bigg\{ \io v\bigg\}^2
	+ \frac{\Gamma_1 m}{a}
	\qquad \mbox{for all } t\in (0,T),
  \eea
  because $|\vp|\le 1$.
  As Lemma \ref{lem2} yields $c_1=c_1(u_0,v_0)>0$ such that
  \bas
	\io v + \io v^2 \le c_1
	\qquad \mbox{for all } t\in (0,T),
  \eas
  estimating
  \bas
	\io \vp^4 u(\ln u - 1)
 	\ge -|\Om| \qquad \mbox{for all } t\in (0,T),
  \eas
  we obtain by inserting (\ref{79.2}) into (\ref{F}) that
  \bas
	\Fp(t) 
	&\ge& - \frac{1}{a} \int_{\supp\vp} u\ln (u+e)
	- \frac{|\Om|}{ae}
	- 2c_1 \|\na\vp\|_{L^\infty(\Om)}^2 \\
	& & - \frac{c_1^2}{4} - \frac{\Gamma_1 m}{a} - |\Om|
	\qquad \mbox{for all } t\in (0,T)
  \eas
  and thereby arrive at the claimed conclusion.
\qed
Now under the assumption in (\ref{H}),
the ill-signed contributions to (\ref{E_evol}) and (\ref{7.1}) can be estimated 
%by drawing on Lemma \ref{lem788}
whenever the number $\del$ in (\ref{H}) is suitably small. 
As detailed in Lemma \ref{lem8}, this will result from the following consequence of Lemma \ref{lem788},
which will independently be used for a second time in Lemma \ref{lem9} later on.
\begin{lem}\label{lem799}
  There exists $\del=\del(\Om)\in (0,1)$ such that whenever (\ref{init}) holds and $T<\infty$, then one can choose
  $\Gamma_6=\Gamma_6(u_0,v_0)>0$ in such a way that if (\ref{H}) is satisfied
  with some blow-up point $x_0\in\B$, time $t_0 \in [0,T)$ and radius $R>0$, then for any $\vp\in C^1(\bom)$ fulfilling $\supp\vp\subset\ov{B}_R(x_0)$ and $0\le\vp\le 1$,
  \be{799.1}
	7 \io \vp^4 u^2
	\le \frac{1}{2} \io \vp^4 \frac{|\na u|^2}{u}
	+ \Gamma_6 \big( \|\na\vp\|_{L^\infty(\Om)}^2 +1\big) \cdot (T-t)^{-\frac{1}{2}}
	\qquad \mbox{for all } t\in (t_0,T).
  \ee
\end{lem}
\proof
  We fix any $\del=\del(\Om)\in (0,1)$ such that with $\Gamma_2$ as in Lemma \ref{lem788},
  \be{799.2}	
	14 \Gamma_2 \del \le \frac{1}{2},
  \ee
  and employ Lemma \ref{lem788} with $\eta:=\frac{1}{14}$ to see that due to (\ref{mass}),
  \bas
	7 \io \vp^4 u^2
	&\le& \frac{1}{2} \io \vp^4 \frac{|\na u|^2}{u}
	+ 21 |\Om| \exp \bigg\{ 14\Gamma_2 \int_{B_R(x_0)\cap\Om} u\ln (u+e)\bigg\}  \\
	& & + c_1 (\|\na\vp\|_{L^\infty(\Om)}^2 + 1)
	\qquad \mbox{for all } t\in (0,T),
  \eas
  where $c_1\equiv c_1(u_0,v_0):=7\Gamma_2 \cdot \big\{ \io u_0\big\}^2$.
  As (\ref{799.2}) warrants that thanks to (\ref{H}) we have
  \bas
	21 |\Om| \exp \bigg\{ 14\Gamma_2 \int_{B_R(x_0)\cap\Om} u\ln (u+e)\bigg\} 
	&\le& 21 |\Om| \exp \Big\{ \frac{1}{2} \ln \frac{T}{T-t}\Big\} \\
	&=& 21|\Om| T^\frac{1}{2} (T-t)^{-\frac{1}{2}}
	\qquad \mbox{for all } t\in (t_0,T),
  \eas
  this implies (\ref{799.1}) if we let
  $\Gamma_6\equiv \Gamma_6(u_0,v_0):=21|\Om| T^\frac{1}{2} + c_1 T^\frac{1}{2}$, because
  $c_1 \le c_1 T^\frac{1}{2} (T-t)^{-\frac{1}{2}}$.
\qed
We are now in the position to refine our knowledge on regularity of the signal gradient in the intended direction.
Still, the following use of the cut-off functions from Lemma \ref{lem21} does not yet rely on the boundary information property in (\ref{21.3}).
\begin{lem}\label{lem8}
  One can find $\del=\del(\Om)\in (0,1)$ with the following property: If $u_0$ and $v_0$ satisfy (\ref{init}) and are such that $T<\infty$,
  and if, for some blow-up point $x_0\in\B$, time $t_0 \in [0,T)$ and radius $R\in (0,R_0)$, the assumption (\ref{H}) holds, then it is possible to choose
  $C=C(x_0,t_0,R,u_0,v_0)>0$ in such a way that
  \be{8.1}
	\int_{B_{\theta^2 R}(x_0)\cap\Om} |\na v(\cdot,t)|^2
	\le C + C \cdot \ln \frac{T}{T-t}
	\qquad \mbox{for all } t\in (t_0,T),
  \ee
  where $R_0$ and $\theta$ are as in Lemma \ref{lem21}.
\end{lem}
\proof
  We take $\del=\del(\Om)>0$ as in Lemma \ref{lem799} and $\vp=\vp^{(x_0,\theta R)}$ as in Lemma \ref{lem21}. We then collect Lemma \ref{lem_E} and Lemma \ref{lem7} as well as the energy-type functionals $\Ep$ and $\Fp$ defined in said lemmas to see,
  \be{8.2}
	y(t):= \Ep(t)
	+ \Fp(t),
	\qquad t\in [0,T),
  \ee
  satisfies
  \bea{8.3}
	& & \hs{-20mm}
	y'(t)
	+ \io \vp^4 \frac{|\na u|^2}{u}
	+ \io \vp^4 |\Del v|^2
	+ \frac{1}{2} \io \vp^4 \Big| \frac{\na u}{\sqrt{u}}-\sqrt{u}\na v\Big|^2 \nn\\
	&\le& \io \vp^4 \na u\cdot\na v
	- 4\io \vp^3 \ln u (\na u-u\na v)\cdot\na\vp \nn\\
	& & + 4 \io \vp^4 u^2
	+ 40 \io \vp^2 |\na v|^2 |\na\vp|^2 \nn\\
	& & + 16 \io \vp^2 u\ln^2 u |\na\vp|^2
	+ 16 \io \vp^2 uv^2 |\na\vp|^2 \nn\\
	& & + 4 \io \vp^2 |\na v|^2 |\na\vp|^2
	\qquad \mbox{for all } t\in (0,T).
  \eea
  Here, an integration by parts shows that due to Young's inequality,
  \bas
	\io \vp^4 \na u\cdot\na v
	&=& - \io \vp^4 u\Del v
	- 4 \io \vp^3 u\na v\cdot\na\vp \\
	&\le& \io \vp^4 |\Del v|^2
	+ \frac{1}{4} \io \vp^4 u^2 \\
	& & + \frac{3}{4} \io \vp^4 u^2
	+ \frac{16}{3} \io \vp^2 |\na v|^2 |\na\vp|^2
	\qquad \mbox{for all } t\in (0,T),
  \eas
  while repeating an application of Young's inequality in the style of Lemma \ref{lem7} we find that
  \bas
	- 4 \io \vp^3 \ln u (\na u-u\na v) \cdot\na\vp
	\le \frac{1}{2} \io \vp^4 \Big| \frac{\na u}{\sqrt{u}}-\sqrt{u}\na v\Big|^2 
	+ 8 \io \vp^2 u \ln^2 u |\na\vp|^2
	\qquad \mbox{for all } t\in (0,T).
  \eas
  Consequently, (\ref{8.3}) implies that
  \bea{8.4}
	y'(t)
	+ \io \vp^4 \frac{|\na u|^2}{u}
	&\le& 5 \io \vp^4 u^2
	+ 24 \io \vp^2 u\ln^2 u |\na\vp|^2 \nn\\
	& & + 16 \io \vp^2 uv^2 |\na\vp|^2
	+ \frac{148}{3} \io \vp^2 |\na v|^2 |\na\vp|^2
	\qquad \mbox{for all } t\in (0,T),
  \eea
  and to further estimate the expressions on the right-hand side herein, we note that $\xi\ln^2 \xi\le \frac{4}{e^2}$ 
  for all $\xi\in (0,1]$, and that $\ln^4 \xi \le \frac{256}{e^4}\xi$ for all $\xi\ge1$;
  therefore, namely, Young's inequality guarantees that
  \bea{8.5}
	\hs{-6mm}
	24 \io \vp^2 u \ln^2 u |\na\vp|^2
	&\le& \frac{96}{e^2} \io \vp^2 |\na\vp|^2 
	+ 24 \int_{\{u>1\}} \vp^2 u\ln^2 u |\na\vp|^2 \nn\\
	&\le& \frac{96}{e^2} \io \vp^2 |\na\vp|^2 
	+ \io \vp^4 u^2
	+ 144 \int_{\{u>1\}} \ln^4 u |\na\vp|^4 \nn\\
	&\le& \frac{96}{e^2} \io \vp^2 |\na\vp|^2 
	+ \io \vp^4 u^2
	+ \frac{9\cdot 2^{12}}{e^4} \io u |\na\vp|^4
	\qquad \mbox{for all } t\in (0,T).
  \eea
  As furthermore, again by Young's inequality,
  \bas
	16 \io \vp^2 uv^2 |\na\vp|^2
	\le \io \vp^4 u^2
	+ 64 \io v^4 |\na\vp|^4
	\qquad \mbox{for all } t\in (0,T),
  \eas
  taking $\Lam$ from Lemma \ref{lem21} and abbreviating $c_1\equiv c_1(R):=\frac{\Lam}{\theta R}$ and
  $c_2\equiv c_2(R,u_0,v_0):=\frac{96 c_1^2 |\Om|}{e^2} + \frac{9\cdot 2^{12} c_1^4}{e^4} \io u_0$
  we infer from (\ref{8.4}), (\ref{8.5}), (\ref{21.4}) and (\ref{mass}) that
  \bea{8.6}
	y'(t)
	+ \io \vp^4 \frac{|\na u|^2}{u}
	&\le& 7 \io \vp^4 u^2
	+ 64 c_1^4 \io v^4 \nn\\
	& & + \frac{148 c_1^2}{3} \int_{B_{\theta R}(x_0)\cap\Om} |\na v|^2
	+ c_2
	\qquad \mbox{for all } t\in (0,T),
  \eea
  because $\supp\vp\subset B_{\theta R}(x_0)$ by Lemma \ref{lem21}.
  Now Lemma \ref{lem799} asserts that thanks to our choice of $\del$ we have
  \bas
	7 \io \vp^4 u^2
	\le \frac{1}{2} \io \vp^4 \frac{|\na u|^2}{u}
	+ c_3(T-t)^{-\frac{1}{2}}
	\qquad \mbox{for all } t\in (t_0,T),
  \eas
  where $c_3\equiv c_3(R,u_0,v_0):=\Gamma_6 \cdot (c_1^2+1)$ with $\Gamma_6$ as provided there.
  Accordingly, (\ref{8.6}) in particular entails that
  \bas
	y'(t) \le h(t):= 64 c_1^2 \io v^4 
	+ \frac{148 c_1^2}{3} \int_{B_{\theta R}(x_0)\cap\Om} |\na v|^2
	+ c_2 + c_3(T-t)^{-\frac{1}{2}}
	\qquad \mbox{for all } t\in (t_0,T),
  \eas
  so that since Lemma \ref{lem3} in conjunction with Lemma \ref{lem2} yields $c_4=c_4(x_0, t_0, R,u_0,v_0)>0$ such that
  \bas
	\int_{t_0}^t h(s) ds \le c_4
	\qquad \mbox{for all } t\in (t_0,T),
  \eas
  we conclude that
  \be{8.7}
	y(t) \le c_5\equiv c_5(x_0,t_0,R,u_0,v_0):= y(t_0)+c_4
	\qquad \mbox{for all } t\in (t_0,T).
  \ee
  It remains to note that in line with (\ref{8.2}) and again (\ref{mass}), taking $\Gamma_5=\Gamma_5(u_0,v_0)$ from Lemma \ref{lem79}
  we know from (\ref{H}) that
  \bas
	y(t) 
	&\ge& - |\Om| + \io \vp^4 |\na v|^2
	- \Gamma_5 \int_{B_R(x_0)\cap\Om} u\ln (u+e)
	- c_1^2 \Gamma_5 - \Gamma_5 \\
	&\ge& \int_{B_{\theta^2 R}(x_0) \cap \Omega} |\na v|^2
	- \Gamma_5 \del \ln\frac{T}{T-t}	- |\Om| - c_1^2 \Gamma_5 - \Gamma_5
	\qquad \mbox{for all } t\in (t_0,T)
  \eas
  as $\vp \equiv 1$ on $B_{\theta^2 R}(x_0)$ by construction.
  Indeed, (\ref{8.7}) therefore implies (\ref{8.1}).
\qed
\subsection{A time-independent local bound for $\int u\ln u dx$. Conclusion}
In a second core step of our analysis, we will now examine the functional $\Ep$ from (\ref{E}) on its own,
independent of its previous role
as a tool to control problematic terms occurring in the analysis of $\Fp$. 
In fact, we now use Lemma \ref{lem8} to suitably bound the possible growth of $\Ep$ during evolution.
Although the functionals studied here essentially deviate from those studied in the previous section only through the absence of the expressions contributed by
$\Fp$ from (\ref{F}), our method of handling associated localization errors will be markedly different in the argument
that leads to the following $L\log L$ bound.
\begin{lem}\label{lem9}
  Let $\del=\del(\Om)$ be as in Lemma \ref{lem799}.
  Then if (\ref{init}) holds and $T<\infty$, and if (\ref{H}) is satisfied with some blow-up point $x_0 \in \B$, time $t_0 \in [0,T)$ and radius $R\in (0,R_0)$,
  one can find 
  $C(x_0, t_0, R,u_0,v_0)>0$ such that
  with $R_0$ and $\theta$ taken from Lemma \ref{lem21}, we have
  \be{9.1}
	\int_{B_{\theta^3 R}(x_0)\cap\Om} u(\cdot,t) \ln u(\cdot,t) 
	\le C
	\qquad \mbox{for all } t\in (0,T).
  \ee
\end{lem}
\proof
  With $\Lam$ as in Lemma \ref{lem21}, we write $c_1\equiv c_1(R):=\frac{\Lam}{\theta^2 R}$ and 
  $c_2\equiv c_2(R):=(\frac{\Lam}{\theta^4 R^2})^2 + c_1^4 + 1$, and then see that the function $\vp:=\vp^{(x_0,\theta^2 R)}$
  obtained there satisfies
  \be{9.2}
	|\na\vp| \le c_1
	\quad \mbox{and} \quad
	|D^2 \vp|^2 + |\na\vp|^4 + 1 \le c_2
	\qquad \mbox{in } \Om.
  \ee
  We moreover recall Lemma \ref{lem2} to fix $c_3=c_3(u_0,v_0)>0$ such that
  \be{9.3}
	\io v^2 \le c_3
	\qquad \mbox{for all } t\in (0,T),
  \ee
  and invoke Lemma \ref{lem8} to infer on the basis of our selection of $\del$ that with some $c_4=c_4(x_0,t_0,R,u_0,v_0)>0$ we have
  \be{9.4}
	\int_{B_{\theta^2 R}(x_0)\cap\Om} |\na v|^2
	\le c_4 \ell(t)
	\qquad \mbox{for all } t\in (t_0,T),
  \ee  
  where we have abbreviated
  \be{9.5}
	\ell(t):=1+\ln\frac{T}{T-t},
	\qquad t\in (0,T).
  \ee
  As already seen in Lemma \ref{lem_E}, we know that the functional $\Ep$ satisfies
  \bas
	& & \hs{-20mm}
	\Ep'(t) + \io \vp^4 \frac{|\na u|^2}{u}
	+ \io \vp^4 |\Del v|^2 \\
	&\le& \io \vp^4 \na u\cdot\na v
	- 4 \io \vp^3 \ln u (\na u- u\na v)\cdot\na\vp \\
	& & + 4\io \vp^4 u^2
	+ 40 \io \vp^2 |\na v|^2 |\na \vp|^2
	\qquad \mbox{for all } t\in (0,T),
  \eas
  where we again integrate by parts and use Young's inequality to estimate
  \bas
	\io \vp^4 \na u\cdot\na v
	&=& - \io \vp^4 u\Del v
	- 4\io \vp^3 u\na v\cdot\na\vp \\
	&\le& \frac{1}{2} \io \vp^4 |\Del v|^2
	+ \frac{1}{2} \io \vp^4 u^2 \\
	& & + \frac{1}{2} \io \vp^4 u^2
	+ 8 \io \vp^2 |\na v|^2 |\na\vp|^2
	\qquad \mbox{for all } t\in (0,T).
  \eas
  Thus, thanks to (\ref{9.2}) and (\ref{9.4}),
  \bea{9.7}
	& & \hs{-20mm}
	\Ep'(t) + \io \vp^4 \frac{|\na u|^2}{u}
	+ \frac{1}{2} \io \vp^4 |\Del v|^2 \nn\\
	&\le& 5 \io \vp^4 u^2
	+ 48 c_1^2 c_4 \ell(t) \nn\\
	& & - 4 \io \vp^3 \ln u (\na u-u\na v)\cdot\na \vp
	\qquad \mbox{for all } t\in (t_0,T).
  \eea
  In contrast to our approach in Lemma \ref{lem8}, however, we now split the rightmost integral herein to see that, by (\ref{9.2}),
  \bea{9.8}
	- 4 \io \vp^3 \ln u (\na u-u\na v)\cdot\na \vp
	&\le& 4c_1 \io \vp^3 |\ln u| \cdot |\na u| \nn\\
	& & + 4c_1 \io \vp^3 u|\ln u| \cdot |\na v|
	\qquad \mbox{for all } t\in (0,T),
  \eea
  where by Young's inequality,
  \be{9.9}
	4c_1 \io \vp^3 |\ln u| \cdot |\na u|
	\le \frac{1}{2} \io \vp^4 \frac{|\na u|^2}{u}
	+ 8c_1^2 \io \vp^2 u\ln^2 u
	\qquad \mbox{for all } t\in (0,T).
  \ee
  Here, using that $\ln^4 \xi \le \frac{256}{e^4} \xi$ for all $\xi\ge 1$ and $\xi\ln^2 \xi \le \frac{4}{e^2}$ for all $\xi\in (0,1)$,
  and writing $m:=\io u_0$,
  we see that due to Young' inequality and (\ref{mass}),
  \bas
	8 c_1^2 \int_{\{u\ge 1\}} \vp^2 u\ln^2 u
	&\le& \io \vp^4 u^2
	+ 16 c_1^4 \int_{\{u\ge 1\}} \ln^4 u \\
	&\le& \io \vp^4 u^2
	+ 16 c_1^4 \cdot \frac{256}{e^4} \io u \\
	&\le& \io \vp^4 u^2
	+ \frac{2^{12} c_1^4 m}{e^4}
	\qquad \mbox{for all } t\in (0,T),
  \eas
  and that
  \bas
	8 c_1^2 \int_{\{u<1\}} \vp^2 u\ln^2 u
	\le \frac{32c_1^2 |\Om|}{e^2}
	\qquad \mbox{for all } t\in (0,T),
  \eas
  whence (\ref{9.9}) entails that
  \be{9.10}
	4c_1 \io \vp^3 |\ln u|\cdot |\na u|
	\le \frac{1}{2} \io \vp^4 \frac{|\na u|^2}{u}
	+ \io \vp^4 u^2
	+ c_5
	\qquad \mbox{for all } t\in (0,T)
  \ee
  with $c_5\equiv c_5(R,u_0,v_0):=\frac{2^{12} c_1^4 m}{e^4} + \frac{32 c_1^2 |\Om|}{e^2}$.\abs
  Next, the crucial last summand in (\ref{9.8}) can be controlled by using the H\"older inequality according to
  \bea{9.11}
	\hs{-10mm}
	4c_1 \io \vp^3 u |\ln u| \cdot |\na v|
	&=& 4c_1 \io \big( \vp^4 |\na v|^4\big)^\frac{1}{4} \cdot \vp^2 u|\ln u| \nn\\
	&\le& 4c_1 \cdot\bigg\{ \io \vp^4 |\na v|^4 \bigg\}^\frac{1}{4} \cdot 
		\bigg\{ \io \vp^\frac{8}{3} u^\frac{4}{3} |\ln u|^\frac{4}{3} \bigg\}^\frac{3}{4}
	\qquad \mbox{for all } t\in (0,T),
  \eea
  where by Lemma \ref{lem33}, (\ref{9.2}),  (\ref{9.3}), (\ref{9.4}) and the fact that $\ell\ge 1$,
  taking $\Gamma_4=\Gamma_4(\Om)$ as introduced there we have
  \bas
	\io \vp^4 |\na v|^4
	&\le& \Gamma_4 \cdot c_4 \ell(t) \io \vp^4 |\Del v|^2
	+ \Gamma_4 \cdot c_2 \cdot \big\{ c_4^2\ell^2(t) + c_3^2 \big\} \\
	&\le& c_6 \ell^2(t) \cdot \bigg\{ 1 + \io \vp^4 |\Del v|^2 \bigg\}
	\qquad \mbox{for all } t\in (t_0,T)
  \eas
  with $c_6\equiv c_6(x_0,t_0,R,u_0,v_0):=\Gamma_4 c_4 + \Gamma_4 c_2(c_4^2+c_3^2)$.
  Therefore, (\ref{9.11}) implies that, again by Young's inequality,
  \bea{9.12}
	4c_1 \io \vp^3 u|\ln u| \cdot |\na v|
	&\le& 2^\frac{1}{4} \cdot 4  c_1 c_6^\frac{1}{4} \ell^\frac{1}{2}(t) \cdot \Bigg\{
	\frac{1}{2} \cdot \bigg\{ 1 + \io \vp^4 |\Del v|^2 \bigg\} \Bigg\}^\frac{1}{4}
	\cdot \bigg\{ \io \vp^\frac{8}{3} u^\frac{4}{3} |\ln u|^\frac{4}{3} \bigg\}^\frac{3}{4} \nn\\
	&\le& \frac{1}{2} \cdot \bigg\{ 1 + \io \vp^4 |\Del v|^2 \bigg\} \nn\\
	& & + c_7 \ell^\frac{2}{3}(t) \io \vp^\frac{8}{3} u^\frac{4}{3} |\ln u|^\frac{4}{3} 
	\qquad \mbox{for all } t\in (t_0,T),
  \eea
  where $c_7\equiv c_7(x_0, t_0,R,u_0,v_0):=(2^\frac{1}{4} \cdot 4c_1 c_6^\frac{1}{4})^\frac{4}{3}$.
  As $\ln^4 \xi \le \frac{256}{e^4} \xi$ for all $\xi\ge 1$ and $\xi|\ln\xi| \le \frac{1}{e}$ for all $\xi\in (0,1)$, by means of Young's inequality and (\ref{mass}) we can here again estimate
  \bas
	c_7 \ell^\frac{2}{3}(t) \int_{\{u\ge 1\}} \vp^\frac{8}{3} u^\frac{4}{3} |\ln u|^\frac{4}{3}
	&=&  c_7\int_{\{u\ge 1\}} \big( \vp^4 u^2 \big)^\frac{2}{3} \cdot (\ell^\frac{1}{2}(t)|\ln u|)^\frac{4}{3} \\
	&\le& \io  \vp^4 u^2 + \frac{256 c_7^3}{e^4} \ell^2(t) \io u\\
  &=&\io  \vp^4 u^2 + c_8 \ell^2(t)
  \eas
  with $c_8 = c_8(x_0, t_0, R, u_0, v_0) :=  \frac{256 c_7^3 m}{e^4}$ and
  \bas
	c_7 \ell^\frac{2}{3}(t) \cdot \int_{\{u<1\}} \vp^\frac{8}{3} u^\frac{4}{3} |\ln u|^\frac{4}{3}
	\le c_7 e^{-\frac{4}{3}} |\Om| \ell^\frac{2}{3}(t)
  \eas
  for all $t\in (t_0,T)$.
  Together with (\ref{9.12}) and (\ref{9.10}) inserted into (\ref{9.8}), this shows that (\ref{9.7}) implies the inequality
  \be{9.13}
	\Ep'(t) + \frac{1}{2} \io \vp^4 \frac{|\na u|^2}{u}
	\le 7 \io \vp^4 u^2
	+ c_{9} \ell^2(t)
	\qquad \mbox{for all } t\in (t_0,T)
  \ee
  with $c_{9}\equiv c_{9}(x_0,R,u_0,v_0):=48c_1^2c_4 +c_5+ \frac{1}{2}+  c_7 e^{-\frac{4}{3}} |\Om| + c_8$.
  Now Lemma \ref{lem799} along with (\ref{9.2}) implies that if we take $\Gamma_6=\Gamma_6(u_0,v_0)$ as selected there, then
  \bas
	7\io \vp^4 u^2
	\le \frac{1}{2} \io \vp^4 \frac{|\na u|^2}{u}
	+ \Gamma_6 \cdot (c_1^2+1) (T-t)^{-\frac{1}{2}}
	\qquad \mbox{for all } t\in (t_0,T).
  \eas
  From (\ref{9.13}) we therefore obtain that
  \bas
	\Ep(t) \le \Ep(t_0) + c_{9} \int_{t_0}^t \ell^2(s) ds
	+ \Gamma_6 \cdot (c_1^2+1) \int_{t_0}^t (T-s)^{-\frac{1}{2}} ds
	\qquad \mbox{for all } t\in (t_0,T),
  \eas
  so that the claim follows upon observing $\ell^2\in L^1((0,T))$ by (\ref{9.5}), and that
  \bas
	\Ep(t) \ge \io \vp^4 u(\ln u-1)
	\ge \int_{B_{\theta^3 R}(x_0)\cap\Om} u\ln u - \frac{|\Omega|}{e}
	- m
	\qquad \mbox{for all } t\in (0,T)
  \eas
  according to Lemma \ref{lem21}, (\ref{mass}) and the fact that 
  $\xi\ln\xi\ge - \frac{1}{e}$ for all $\xi \ge 0$.
\qed
In concluding Theorem \ref{theo11} from the above, we will make use of the following characterization of blow-up points
that was achieved already in \cite[Proposition 4.2]{NSS2000}.
\begin{prop}\label{prop10}
  Assume (\ref{init}), and suppose that in Proposition \ref{prop0} we have $T<\infty$. Then whenever $x_0\in\bom$ is such that
  \be{10.1}
	\sup_{t\in (0,T)} \int_{B_{r_0}(x_0)\cap\Om} u(\cdot,t) \ln u(\cdot,t) <\infty
  \ee
  for some $r_0>0$, we have
  \bas
	\sup_{t\in (0,T)} \|u(\cdot,t)\|_{L^\infty(B_r(x_0)\cap \Om)} < \infty
	\qquad \mbox{for all } r\in (0,r_0);
  \eas
  in particular, it then follows that $x_0\not\in\B$.
\end{prop}
Accordingly, Lemma \ref{lem9} indeed has essentially established our main result:\abs
\proofc of Theorem \ref{theo11}. \quad
  We let $\del=\del(\Om)$ be as in Lemma \ref{lem799}, and pick any $\del_0=\del_0(\Om)\in (0,\del)$.
  Then supposing (\ref{11.1}) to be false for some $x_0\in\B$ and $R>0$, writing $B:=B_R(x_0)\cap\Om$ we can rely on (\ref{mass})
  and use that $-\xi\ln\xi\le\frac{1}{e}$ for all $\xi>0$ in estimating
  \bas
	\int_B u\ln (u+e)
	&=& \int_B u\ln u
	+ \int_{B\cap\{u<1\}} u\ln (u+e)
	+ \int_{B\cap\{u>1\}} u \ln \Big(1+\frac{e}{u}\Big)
	- \int_{B\cap\{u<1\}} u\ln u \\
	&\le& \int_B u\ln u
	+ \ln (1+e) \int_{B\cap\{u<1\}} u + \ln (1+e) \int_{B\cap\{u\ge 1\}} u + \frac{|\Om|}{e} \\
	&\le& \int_B u\ln u
	+ \ln (1+e) \io u_0 + \frac{|\Om|}{e} 
	\qquad \mbox{for all } t\in (0,T),
  \eas
  so that, by hypothesis,
  \bas
	\limsup_{t\nearrow T} \frac{1}{\ln\frac{T}{T-t}} \int_B u\ln (u+e)
	\le \limsup_{t\nearrow T} \frac{1}{\ln\frac{T}{T-t}} \int_B u\ln u
	\le \del_0.
  \eas
  Since $\del_0<\del$, we could thus find $t_0\in [0,T)$ such that
  \bas
	\int_B u\ln (u+e) \le \del\cdot \ln\frac{T}{T-t}
	\qquad \mbox{for all } t\in (t_0,T).
  \eas
  Since thus we would obtain that (H) holds with some $R>0$ which, after diminishing if necessary, can clearly be assumed to be such that
  $R<R_0$ with $R_0$ as provided by Lemma \ref{lem21},
  we could draw on Lemma \ref{lem9} to infer that
  (\ref{10.1}) would be satisfied with $r_0:=\theta^3 R$ and $\theta$ as in Lemma \ref{lem21}.
  This, however, would be incompatible with the outcome of Proposition \ref{prop10}.
\qed
Our claim concerning lower bounds on localized spatial $L^p$ norms, finally, can be obtained from this 
by straightforward interpolation:\abs
\proofc of Corollary \ref{cor12}. \quad
  To see that the claimed conclusion holds with 
  \be{12.01}
	C\equiv C(u_0,v_0):=e^{-\frac{1}{me}} \cdot \min\big\{ T^\frac{\del_0}{m} \, , \, 1\big\},
  \ee
  we fix $x_0\in\B$ and $R>0$, and employ Theorem \ref{theo11} to find $(t_k)_{k\in\N} \subset (0,T)$ 
  such that $t_k\nearrow T$ as $k\to\infty$ and
  \be{12.3}
	\int_{B_R(x_0)\cap\Om} u(\cdot,t_k)\ln u(\cdot,t_k) \ge \del_0\cdot\ln\frac{T}{T-t_k}
	\qquad \mbox{for all } k\in\N.
  \ee
  When $p=\infty$, in view of (\ref{mass}) this directly entails that 
  \bas
	\del_0\cdot\ln\frac{T}{T-t_k}
	\le m\cdot\ln \|u(\cdot,t_k)\|_{L^\infty(B_R(x_0)\cap\Om)}
  \eas
  and hence
  \bas
	(T-t_k)^\frac{\del_0}{m} \|u(\cdot,t_k)\|_{L^\infty(B_R(x_0)\cap\Om)} \ge T^\frac{\del_0}{m}
  \eas
  for all $k\in\N$, so that the choice in (\ref{12.01}) clearly ensures validity of (\ref{12.1}) in this case.\abs
  If $p\in (1,\infty)$, then writing $m_R(t):=\int_{B_R(x_0)\cap\Om} u(\cdot,t)$ for $t\in (0,T)$ and using the Jensen inequality 
  as well as (\ref{mass}) we can estimate
  \bea{12.2}
	\int_{B_R(x_0)\cap\Om} u\ln u
	&=& \frac{m_R(t)}{p-1} \int_{B_R(x_0)\cap\Om} \ln (u^{p-1}) \cdot \frac{u}{m_R(t)} dx \nn\\
	&\le& \frac{m_R(t)}{p-1} \ln \bigg\{ \frac{1}{m_R(t)} \int_{B_R(x_0)\cap\Om} u^p \bigg\} \nn\\
	&=& \frac{m_R(t)}{p-1} \ln \int_{B_R(x_0)\cap\Om} u^p
	- \frac{1}{p-1} m_R(t) \ln m_R(t) \nn\\
	&\le& \frac{m_R(t)}{p-1} \ln \int_{B_R(x_0)\cap\Om} u^p
	+ \frac{1}{(p-1)e}
	\qquad \mbox{for all $t\in (0,T)$},
  \eea
  so that since the right-hand side in (\ref{12.3}) diverges to $+\infty$ as $k\to\infty$ we firstly obtain from (\ref{12.2}) that 
  necessarily
  $\int_{B_R(x_0)\cap\Om} u^p(\cdot,t_k)\to + \infty$ as $k\to\infty$, whence passing to a subsequence if necessary we may assume
  that $\int_{B_R(x_0)\cap\Om} u^p(\cdot,t_k)\ge 1$ for all $k\in\N$.
  We may therefore use that $m_R(t)\le m$ for all $t\in (0,T)$ to infer from (\ref{12.2}) and (\ref{12.3}) that
  \bas
	\del_0 \cdot \ln \frac{T}{T-t_k}
	\le \frac{m}{p-1} \ln \int_{B_R(x_0)\cap\Om} u^p(\cdot,t_k)
	+ \frac{1}{(p-1)e}
	\qquad \mbox{for all } k\in\N,
  \eas
  that is,
  \bas
	\int_{B_R(x_0)\cap\Om} u^p(\cdot,t_k) \ge e^{-\frac{1}{me}} \cdot\Big(\frac{T}{T-t_k}\Big)^\frac{(p-1)\del_0}{m}
	\qquad \mbox{for all } k\in\N
  \eas
  or, equivalently,
  \bas
	(T-t_k)^{\frac{p-1}{p} \cdot \frac{\del_0}{m}} \|u(\cdot,t_k)\|_{L^p(B_R(x_0)\cap\Om)}
	\ge e^{-\frac{1}{pme}} T^{\frac{p-1}{p} \cdot\frac{\del_0}{m}}
	\qquad \mbox{for all } k\in\N.
  \eas
  Since $e^{-\frac{1}{pme}} T^{\frac{p-1}{p} \cdot\frac{\del_0}{m}} \ge C$ by (\ref{12.01}), we thereby obtain (\ref{12.1})
  also in this case.
\qed

\section*{Declarations}
{\bf Funding.} \quad
The authors acknowledge support of the Deutsche Forschungsgemeinschaft (Project No.~462888149).\abs
{\bf Conflict of interest statement.} \quad
The authors declare that they have no conflict of interest.\abs
%and that they have no relevant financial or non-financial interests to disclose.\abs
%
{\bf Data availability statement.} \quad
Data sharing is not applicable to this article as no datasets were
generated or analyzed during the current study.


\begin{thebibliography}{99}
%
\bibitem{agudelo_pistoia}
  \sc Agudelo, O., Pistoia, A.:
  \it Boundary concentration phenomena for the higher-dimensional Keller-Segel system. 
  \rm Calc. Var. Partial Differ. Equ. {\bf 55}, 132 (2016)
\bibitem{xueli_bai}
  \sc Bai, X., Zhou, M.:
  \it Exact blow-up profiles for the parabolic-elliptic Keller-Segel system in dimensions $N\ge 3$.
  \tt arXiv:2406.07201 \rm
\bibitem{biler}
  \sc Biler, P.: \it Local and global solvability of some parabolic systems modelling chemotaxis.
  \rm Adv. Math. Sci. Appl. {\bf 8}, 715-743 (1998)
\bibitem{BHN}
  \sc Biler, P., Hebisch, T., Nadzieja, T.: 
  \it The Debye system: existence and large time behavior of solutions.
  \rm Nonlinear Anal. {\bf 23}, 1189-1209 (1994)
\bibitem{bonheure_JMPA}
  \sc Bonheure, D., Casteras, J.-B., F\"oldes, J.:
  \it Singular radial solutions for the Keller-Segel equation in high dimension.
  \rm J. Math. Pures Appl. {\bf 134}, 204-254 (2020)
\bibitem{delpino_arxiv}
  \sc Buseghin, F., D\'avila, J., Del Pino, M., Musso, M.:
  \it Existence of finite time blow-up in Keller-Segel system.
  \tt arXiv:2312.01475 \rm
\bibitem{cao_small}
  \sc Cao, X.: 
  \it Global bounded solutions of the higher-dimensional Keller-Segel system under smallness conditions in optimal spaces.
  \rm Discr.~Cont.~Dyn.~Syst.~A {\bf 35}, 1891-1904 (2015)
\bibitem{cao_arxiv}
  \sc Cao, X.:
  \it An interpolation inequality and its application in Keller-Segel model.
  \tt arXiv:1707.09235 \rm
\bibitem{collot_JFA2023}
  \sc Collot, C., Ghoul, T.-E., Masmoudi, N., Nguyen, V.~T.:
  \it Collapsing-ring blowup solutions for the Keller-Segel system in three dimensions and higher. 
  \rm J. Funct. Anal. {\bf 285}, 110065 (2023)
%\bibitem{collot_CPAM2022}
%  \sc Collot, C., Ghoul, T.-E., Masmoudi, N., Nguyen, V.-T.:
%  \it Refined description and stability for singular solutions of the 2D Keller-Segel system. 
%  \rm Commun. Pure Appl. Math. {\bf 75}, 1419-1516 (2022)
\bibitem{DDelPDMW}
  \sc D\'avila, J., del Pino, M., Musso, M., Wei, J.:
  \it Existence and stability of infinite time blow-up in the Keller-Segel system.
  \rm Arch. Ration. Mech. Anal. {\bf 248}, 61 (2024)
\bibitem{delpino_JEMS2014}
  \sc Del Pino, M., Mahmoudi, F., Musso, M.:
  \it Bubbling on boundary submanifolds for the Lin-Ni-Takagi problem at higher critical exponents.
  \rm J. Eur. Math. Soc. {\bf 16}, 1687-1748 (2014)
\bibitem{delpino_pistoia_vaira}
  \sc Del Pino, M., Pistoia, A., Vaira, G.:
  \it Large mass boundary condensation patterns in the stationary Keller-Segel system. 
  \rm J. Differential Equations {\bf 261}, 3414-3462 (2016)
\bibitem{ghoul_CPAM2018}
  \sc Ghoul, T.-E., Masmoudi, N.:
  \it Minimal mass blowup solutions for the Patlak-Keller-Segel equation.
  \rm Commun. Pure Appl. Math. {\bf 71}, 1957-2015 (2018)
\bibitem{GT}
  \sc Gilbarg, D., Trudinger, N.S.: 
  \it Elliptic Partial Differential Equations of Second Order.
  \rm Springer-Verlag, Berlin/Heidelberg 2001
\bibitem{glogic}
  \sc Glogi\'c, I., Sch\"orkhuber, B.:
  \it Stable singularity formation for the Keller-Segel system in three dimensions. 
  \rm Arch. Ration. Mech. Anal. {\bf 248}, 4 (2024); correction ibid. {\bf 248}, 57 (2024)
\bibitem{HV}
  \sc Herrero, M.~A., Vel\'azquez, J.~J.~L.: 
  \it A blow-up mechanism for a chemotaxis model.
  \rm Ann.~Scu.~Norm.~Sup.~Pisa Cl.~Sci. {\bf 24}, 633-683 (1997)
\bibitem{horstmann}
  \sc Horstmann, D.: \it 
  From 1970 until present: The Keller-Segel model in chemotaxis and its
  consequences I.
  \rm Jahresberichte DMV {\bf 105} (3), 103-165 (2003)
\bibitem{horstmann_wang}
  \sc Horstmann, D., Wang, G.:
  \it Blow-up in a chemotaxis model without symmetry assumptions.
  \rm Eur.~J.~Appl.~Math. {\bf 12}, 159-177 (2001)
\bibitem{horstmann_win}
  \sc Horstmann, D., Winkler, M.: \it Boundedness vs. blow-up in a chemotaxis system. 
  \rm J.~Differential Equations {\bf 215} (1), 52-107 (2005)
\bibitem{jaeger_luckhaus}
  \sc J\"ager, W., Luckhaus, S.: 
  \it On explosions of solutions to a system of partial differential equations modelling chemotaxis.
  \rm Trans.~Am.~Math.~Soc. {\bf 329}, 819-824 (1992)
\bibitem{KS}
  \sc Keller, E.~F., Segel, L.~A.: \it Initiation of slime mold aggregation 
  viewed as an instability.
  \rm J.~Theor.~Biol. {\bf 26}, 399-415 (1970)
\bibitem{lee_manifold_book} 
  \sc Lee, J. M.:
  \it Introduction to smooth manifolds.
  \rm Springer, New York, 2013  
\bibitem{mizo}
  \sc Mizoguchi, N.:
  \it Refined asymptotic behavior of blowup solutions to a simplified chemotaxis system.
  \rm Commun. Pure Appl. Math. {\bf 75}, 1870-1886 (2022)
\bibitem{mizoguchi_souplet}
  \sc Mizoguchi, N., Souplet, Ph.: \it Nondegeneracy of blow-up points for the parabolic Keller-Segel system. 
  \rm Ann.~Inst.~H.~Poincar\'e Anal.~Non Lin\'eaire {\bf 31}, 851-875 (2014)
\bibitem{nagai2001}
  \sc Nagai, T.: \it Blowup of Nonradial Solutions to Parabolic-Elliptic Systems Modeling
  Chemotaxis in Two-Dimensional Domains.
  \rm J.~Inequal.~Appl. {\bf 6}, 37-55 (2001)
\bibitem{NSS2000}
  \sc Nagai, T., Senba, T., Suzuki, T.: \it Chemotactic collapse in a 
  parabolic system of mathematical biology.
  \rm Hiroshima Math.~J. {\bf 30}, 463-497 (2000)
\bibitem{NSY}
  \sc Nagai, T., Senba, T., Yoshida, K.: \it Application of the Trudinger-Moser inequality to a parabolic system of chemotaxis.
  \rm Funkc.~Ekvacioj, Ser.~Int. {\bf 40}, 411-433 (1997)
\bibitem{nanjundiah}
  \sc Nanjundiah, V.: \it Chemotaxis, signal relaying and aggregation morphology.
  \rm J.~Theor.~Biol. {\bf 42}, 63-105 (1973)
\bibitem{zaag}
  \sc Nguyen, V.T., Nouaili, N., Zaag, H.:
  \it Construction of type I-Log blowup for the Keller-Segel system in dimensions 3 and 4.
  \tt arXiv:2309.13932 \rm
\bibitem{ni_takagi_DMJ1993}
  \sc Ni, W.-M., Takagi, I.:
  \it  Locating the peaks of least-energy solutions to a semilinear Neumann problem. 
  \rm Duke Math. J. {\bf 70}, 247-281 (1993)
\bibitem{osaki_yagi}
  \sc Osaki, K., Yagi, A.: 
  \it Finite dimensional attractor for one-dimensional Keller-Segel equations,
  \rm Funkcialaj Ekvacioj {\bf 44}, 441 - 469 (2001)
\bibitem{schweyer}
  \sc Schweyer, R.: \it Stable blow-up dynamic for the parabolic-parabolic Patlak-Keller-Segel model.
  \rm {\tt arXiv:1403.4975}
\bibitem{senba_NA2007}
  \sc Senba, T.:  \it Type II blowup of solutions to a simplified Keller-Segel system in two dimensional domains.
  \rm Nonlinear Anal. {\bf 66}, 1817-1839 (2007)
\bibitem{senba_suzuki_ADE2001}
  \sc Senba, T., Suzuki, T.: \it Chemotactic collapse in a parabolic-elliptic system of mathematical biology.
  \rm Adv.~Differential Eq. {\bf 6}, 21-50 (2001)
\bibitem{senba_suzuki_ADE2003}
  \sc Senba, T., Suzuki, T.: \it Blowup behavior of solutions to rescaled J\"ager-Luckhaus system.
  \rm Adv.~Differential Eq. {\bf 8}, 787-820 (2003)
\bibitem{souplet_win}
  \sc Souplet, Ph., Winkler, M.:
  \it Blow-up profiles for the parabolic-elliptic Keller-Segel system in dimensions $n\ge 3$.
  \rm Commun.~Math.~Phys. {\bf 367}, 665-681 (2019)
\bibitem{suzuki_book}
  \sc Suzuki, T.: 
  \it Free Energy and Self-interacting Particles.
  \rm Progress in Nonlinear Differential Equations and Their Applications 62, Birkh\"auser, Boston, 2005
\bibitem{win_JMPA}
  \sc Winkler, M: 
  \it Finite-time blow-up in the higher-dimensional parabolic-parabolic Keller-Segel system.
  \rm Journal de Math\'ematiques Pures et Appliqu\'ees {\bf 100}, 748-767 (2013), {\tt arXiv:1112.4156v1}
\bibitem{win_SIMA}
  \sc Winkler, M.: 
  \it Small-mass solutions in the two-dimensional Keller-Segel system coupled to the Navier-Stokes equations.
  \rm SIAM Journal on Mathematical Analysis {\bf 52}, 2041-2080 (2020)
\bibitem{win_NON}
  \sc Winkler, M.:
  \it Single-point blow-up in the Cauchy problem for the higher-dimensional Keller-Segel system.
  \rm Nonlinearity {\bf 33}, 5007-5048 (2020)
\bibitem{win_JEMS}
  \sc Winkler, M.:
  \it Can diffusion degeneracies enhance complexity in chemotactic aggregation? 
  Finite-time blow-up on spheres in a quasilinear Keller-Segel system. 
  \rm J.~Eur.~Math.~Soc., to appear
\bibitem{wloka_pde_book} 
  \sc Wloka, J.:
  \it Partial differential equations.
  \rm Cambridge University Press, Cambridge, 1987 
%
\end{thebibliography}
\end{document}